\newcommand{\R}{\mathds R}

\newcommand{\Z}{\mathds Z}
\newcommand{\N}{\mathds N}

\newcommand{\llangle}{\langle\!\!\langle}
\newcommand{\rrangle}{\rangle\!\!\rangle}

\newcommand{\Ddt}{\tfrac{\mathrm D}{\mathrm dt}}
\newcommand{\ddt}{\tfrac{\mathrm d}{\mathrm dt}}

\newcommand{\Dim}{\mathrm{dim}}
\newcommand{\Ker}{\mathrm{Ker}}

\documentclass[twoside,a4paper,10pt]{amsart}

\usepackage{times}
\usepackage{dsfont}

\renewcommand{\contentsline}[3]{\csname new#1\endcsname{#2}{#3}}
\newcommand{\newchapter}[2]{\bigskip\hbox to \hsize{\vbox{\advance\hsize by -.5cm\baselineskip=12pt\parfillskip=0pt\leftskip=2cm\noindent\hskip -2cm #1\leaders\hbox{.}\hfil\hfil\par}$\,$#2\hfil}}
\newcommand{\newsection}[2]{\medskip\hbox to \hsize{\vbox{\advance\hsize by -.5cm\baselineskip=12pt\parfillskip=0pt\leftskip=2.5cm\noindent\hskip -2cm #1\leaders\hbox{.}\hfil\hfil\par}$\,$#2\hfil}}
\newcommand{\newsubsection}[2]{\medskip\hbox to \hsize{\vbox{\advance\hsize by -.5cm\baselineskip=12pt\parfillskip=0pt\leftskip=3.5cm\noindent\hskip -2cm #1\leaders\hbox{.}\hfil\hfil\par}$\,$#2\hfil}}

\numberwithin{equation}{section}

\title[Genericity of nondegenerate critical points]{Genericity of nondegenerate critical points \\
and Morse geodesic functionals}
\author[L. Biliotti]{Leonardo Biliotti}
\address{Dipartimento di Matematica\hfill\break\indent Universit\`a
di Parma\hfill\break\indent Viale G. Usberti 53/A\hfill\break\indent 43100 Parma, Italy}
\email{leonardo.biliotti@unipr.it}

\author[M. A. Javaloyes]{Miguel Angel Javaloyes}
\address{Departamento de Geometr\'{\i}a y Topolog\'{\i}a.\hfill\break\indent
 Facultad de Ciencias, Universidad de Granada.\hfill\break\indent
 Campus Fuentenueva s/n, 18071 Granada, Spain}
 \email{ma.javaloyes@gmail.es}
\thanks{The second author was partially supported by
Regional J. Andaluc\'{\i}a Grant P06-FQM-01951 and by Spanish MEC Grant MTM2007-64504.
The third author is sponsored by Capes, Brasil, Grant BEX 1509-08-0.}

\author[P.\ Piccione]{Paolo Piccione}
\address{Departamento de Matem\'atica,\hfill\break\indent
Universidade de S\~ao Paulo, \hfill\break\indent Rua do Mat\~ao
1010,\hfill\break\indent CEP 05508-900, S\~ao Paulo, SP, Brazil}
\curraddr{Department of Mathematics, \hfill\break\indent
University of Murcia, Campus de Espinardo\hfill\break\indent
30100 Espinardo, Murcia, \hfill\break\indent Spain}
\email{piccione.p@gmail.com}

\subjclass[2000]{57R45, 57R70, 57N75, 58E10}

\date{November 27th, 2008}

\begin{document}


\theoremstyle{plain}\newtheorem*{teon}{Theorem}
\theoremstyle{definition}\newtheorem*{defin*}{Definition}
\theoremstyle{plain}\newtheorem{teo}{Theorem}[section]
\theoremstyle{plain}\newtheorem{prop}[teo]{Proposition}
\theoremstyle{plain}\newtheorem{lem}[teo]{Lemma}
\theoremstyle{plain}\newtheorem*{lem-n}{Lemma}
\theoremstyle{plain}\newtheorem{cor}[teo]{Corollary}
\theoremstyle{definition}\newtheorem{defin}[teo]{Definition}
\theoremstyle{remark}\newtheorem{rem}[teo]{Remark}
\theoremstyle{plain} \newtheorem{assum}[teo]{Assumption}
\swapnumbers
\theoremstyle{definition}\newtheorem{example}{Example}


\begin{abstract}
We consider a family of variational problems on a Hilbert manifold parameterized by
an open subset of a Banach manifold, and we discuss the genericity of the nondegeneracy
condition for the critical points. Using classical techniques,
we prove an abstract genericity result that employs the infinite dimensional Sard--Smale
theorem, along the lines of an analogous result of B. White \cite{Whi}.
Applications are given by proving the genericity of
metrics without degenerate geodesics between fixed endpoints in general (non compact)
semi-Riemannian manifolds, in orthogonally split semi-Riemannian manifolds and in
globally hyperbolic Lorentzian manifolds. We discuss the genericity property also in
stationary Lorentzian manifolds.
\end{abstract}

\maketitle
\tableofcontents

\begin{section}{Introduction}
Generic properties of flows, especially of Riemannian geodesic flows, are a classical topic
in the theory of dynamical systems and in calculus of variations, with important contributions
by many authors. A well known result of the area is the so-called \emph{bumpy metric theorem}, originally
formulated by Abraham \cite{Abr}, and proved in detail by Anosov \cite{Ano}, which states that
the bumpy Riemannian metrics over a given compact manifold form a generic set. Recall that a metric is
bumpy when all its closed geodesics are nondegenerate.
Recently, B. White \cite{Whi} has proven a nice formulation of the bumpy metric theorem in the context of minimal
immersions; more precisely, given a compact manifold $M$ and a complete Riemannian manifold $(N,h)$, with $\Dim(M)<\Dim(N)$,
then
the Riemannian metrics $h$ on $N$ such that every minimal \emph{embedding} $\phi:(M,g)\to(N,h)$
is nondegenerate form a generic set. In the case $M=S^1$, White's theorem does not reproduce exactly the
bumpy metric theorem, in that the result does not guarantee that iterates of a given closed geodesic,
which are not embeddings, are also nondegenerate. A key point in the proof of this result, which has a variational
nature, is that the Jacobi differential operator arising from the second variational formula of the
area functional is a self-adjoint Fredholm operator.
Inspired by White's result, the goal of the present paper is to initiate a study of
generic properties of geodesics in \emph{semi-Riemannian} manifolds, i.e., in manifolds endowed with a non positive definite nondegenerate
metric tensor. At the present stage, this is a totally unexplored field.
Motivations for the interest in such kind of dynamical systems are obviously related
to Lorentzian geometry and General Relativity, to which this paper is ultimately devoted, but also
to Morse theory, as explained below, and to the general theory of semi-Riemannian manifolds.
As a starting point for our theory, we
consider the case of \emph{fixed endpoints} geodesics in semi-Riemannian manifolds. We set ourself the
task of determining whether the set of semi-Riemannian metrics on a fixed manifold $M$ that:
\begin{itemize}
\item have fixed index;
\item belong to some specific class, such as orthogonally split, globally hyperbolic,  or are conformal to some given
metric;
\item make any  two arbitrarily fixed distinct points
non conjugate along any geodesic,
\end{itemize}
is generic. One should observe that in the non positive definite case, the Jacobi differential operator
is not self-adjoint, or even normal, but the index form (i.e., the second variation of the geodesic
action functional) along a given geodesic is represented by a self-adjoint Fredholm operator.
Recall that, given $p,q\in M$, the nonconjugacy property above relatively to some semi-Riemannian metric $g$
on $M$ is equivalent to the fact that the $g$-geodesic action functional
$\Omega_{p,q}\ni\gamma\mapsto\frac12\int_0^1g(\dot\gamma,\dot\gamma)\,\mathrm dt\in\R$,
defined on the Hilbert manifold $\Omega_{p,q}$ of all curves of Sobolev class $H^1$ in $M$
joining $p$ and $q$, is a \emph{Morse function}. Standard Morse theory does not apply to
the semi-Riemannian geodesic action functional, due to the fact that in the non positive
definite case all its critical points have infinite Morse index. Recent developments of
Morse theory, mostly due to the work of Abbondandolo and
Majer (see \cite{AbbMej2, AbbMej}) have shown that, under suitable assumptions, one
can construct a doubly infinite chain complex (Morse--Witten complex) out of the critical
points of a strongly indefinite Morse functional, using the dynamics of the gradient flow.
The Morse relations for the critical points are obtained by computing the homology of this complex,
which in the standard Morse theory is  isomorphic to the singular homology of the base manifolds.
Such computation is one of the central and highly non trivial issues of the theory.
Remarkably, Abbondandolo and
Majer have also shown that this homology is stable by ``small'' perturbations, so that in several
concrete examples one can reduce its computation to a simpler case. This occurs for instance
in the case of the  geodesic action functional in a globally hyperbolic Lorentzian manifold,
in which case the homology of the Morse--Witten complex is stable by small $C^0$ perturbations
of the metric. Thus, it becomes a relevant issue to discuss under which circumstances a given
metric tensor can be perturbed in a given class in such a way that the nondegenericity property
for its geodesics between two prescribed points is preserved. This problem is the original
motivation for the results developed in this paper; we  basically give an affirmative
answer to the genericity questions posed above, with three remarkable exceptions that will be discussed
below.

The idea for proving the genericity of the nondegeneracy property for the critical points of
a family of functionals, which follows a standard transversality approach (see the classical
reference \cite[Chapter 4]{AbrRob}, or the more recent \cite[Section~2.11]{AbboMaj3}),
is the following. Assume that one is given a Hilbert manifold $Y$, and
a family of functionals $f_x:Y\to\R$ parameterized by
points $x$ in an open subset $\mathcal A$ of a Banach space $X$. In the geodesic case, $Y$
is the Hilbert manifold $\Omega_{p,q}(M)$ of curves between two fixed points
in a manifold $M$, $X$ is the space of $(0,2)$-symmetric tensors on $M$,
and $\mathcal A$ is the open set of nondegenerate tensors having a fixed index.
Then, one consider the set of pairs \[\mathfrak M=\big\{(x,y):\text{$y$ is a critical point of $f_x$}\big\},\]
which under suitable assumptions has the following important properties:
\begin{itemize}
\item $\mathfrak M$ is an embedded submanifold of the product $X\times Y$;
\item the projection $\Pi:\mathfrak M\to X$ onto the first factor is a smooth nonlinear Fredholm map of index $0$;
\item the critical values of $\Pi$ are precisely the set of $x\in\mathcal A$ such that $f_x$ has some degenerate critical
point in $Y$.
\end{itemize}
Thus, the genericity of nondegenerate critical points is reduced to the question of regular values
of a Fredholm map, to which Sard--Smale theorem gives a complete answer. In order to make this setup working,
one needs some regularity and Fredholmness assumptions, plus a certain transversality assumption
that in the geodesic case reduces to the existence of some special tensors on the underlying manifold.
\smallskip

There are three cases in which the genericity property of nondegenerate geodesics either fails,
or cannot be proven with the techniques of this paper. First,  perturbations
in a given conformal class are insufficient to eliminate degeneracies of \emph{lightlike} geodesics.
In fact, every conformal perturbation of a semi-Riemannian metric preserves lightlike pregeodesics
and their conjugate points, so that nondegeneracy is \emph{not} generic in a given conformal class.
The second, and more intriguing, point that deserves further attention is the case
of geodesics with the same initial and endpoint and, more specifically, the case
of periodic geodesics.  Note that in the case of periodic geodesics, the notion of nondegeneracy
has to be modified, due to the fact that in the periodic case the tangent field to a geodesic
is always in the kernel of the index form. Every periodic geodesic produces a countable
number of distinct critical points of the action functional by iteration. In order to
develop Morse theory, one clearly needs to have nondegeneracy of all this iterates, which
amounts to saying that the linearized Poincar\'e map along the given geodesic should not
have any (complex) roots of unity in its spectrum. Due to some technical reasons,
the metric perturbations studied in this paper fail to produce the desired result
in the case of a $1$-periodic geodesic $\gamma$
some of whose iterates $\gamma^k$ admits a nontrivial periodic Jacobi field $J$ satisfying
$\sum_{j=1}^kJ_{t+j}=0$ for all $t$. Examples of this situation can be constructed easily, for instance
by considering periodic geodesics on a flat M\"obius strip. Roughly speaking, the field $V_t=\sum_{j=1}^kJ_{t+j}$
indicates in which direction the metric should be \emph{stretched} in order to destroy the
degeneracy produced by the Jacobi field $J$.
Due to this problem, all our genericity results use the (probably unnecessary)
assumption that the endpoints should be distinct.
It is curious to observe that, also under this assumption, one does not avoid having to
deal with \emph{portions} of periodic geodesics (see Lemma~\ref{thm:selfintersections}), but this
case is treated with a little ``parity'' trick.
We conjecture that most of the results of this paper should hold also in the case of periodic
geodesic (in the Riemannian case this is established in \cite{Abr} and \cite{Ano}), but
the proof should be based on dynamical arguments, rather than variational.
The third situation where the transversality condition is not satisfied, and thus the genericity
of metrics with nondegenerate geodesics cannot be deduced by the theory in the present
paper, is the case of stationary Lorentzian manifolds. We will show with an explicit example that, in
the class of stationary metrics on a manifold $M$ having a prescribed vector field
$Y\in\mathfrak X(M)$ as timelike Killing vector field, the transversality condition fails
to hold along a degenerate geodesic which is an integral line of $Y$.

\smallskip

We will now give a detailed technical description of the material discussed in the paper,
with a few additional remarks.
In Section~\ref{sec:preliminaries} we fix notations and discuss a few
preliminary results involving the functional analytical setup and the geometrical setup of the paper.
In the functional analytical part we determine a criterion for the surjectivity (Lemma~\ref{thm:lemkercompl})
and a criterion for existence of a closed complement
to the kernel (Proposition~\ref{thm:lemsurj}) of the direct sum of two bounded linear operators.
These are used to determine transversality to the zero section of a cotangent bundle
for the partial derivative of a map defined
on product spaces (Proposition~\ref{thm:transversality}).
The main result of the geometrical part is Lemma~\ref{thm:constrloc-strong}, that gives the existence
of (global) sections of a vector bundle endowed with a connection, whose value and covariant derivative
have been prescribed along a small immersed curve in the base.
In Section~\ref{sec:abstrgenresult} we consider the abstract setup of a family of smooth functionals
over a Hilbert manifold, parameterized by points of (an open subset of) a Banach manifold.
The central result, Corollary~\ref{thm:genericnondegeneracy}, uses a certain transversality assumption
(see formula~\eqref{eq:transvcond}) to characterize the Morse functionals in the family as regular
values of a nonlinear Fredholm map, yielding the desired genericity result via Sard--Smale theorem.
The main idea and the proof of Corollary~\ref{thm:genericnondegeneracy} follow closely B. White's
arguments in the abstract setup of \cite[Section~1]{Whi}.
In Section~\ref{sec:morsegeofunctionals} we apply Corollary~\ref{thm:genericnondegeneracy} to
the fixed endpoint geodesic problem in several contexts. We will first consider (Subsections~\ref{sub:metrics} and \ref{sub:semiRiemgeneric})
the case of general semi-Riemannian metrics on an arbitrarily fixed manifold, possibly non compact.
When dealing with a non compact manifold $M$, there is no canonical Banach space structure on the space
of tensors on $M$, and in particular there is no way of describing semi-Riemannian metric tensors
as an open subset of a Banach space. Note that Sard--Smale theorem uses a Banach space structure in
an essential way. One way to induce a Banach space norm in the space of tensors would be to use
an auxiliary complete Riemannian metric $g_{\mathrm R}$ on $M$, and then considering tensors of class
$C^k$ on $M$ whose first $k$ (covariant) derivatives have bounded $g_{\mathrm R}$-norm (see Example~\ref{exa:exaWtBstf}).
However, a more general genericity statement is obtained by considering the notion of \emph{$C^k$-Whitney type}
Banach space of tensors on $M$, which is introduced in Subsection~\ref{sub:metrics}.
A Banach space of tensors $\mathcal E$ is said to be of $C^k$-Whitney type if it contains all tensors of class $C^k$ with compact support
(these are used in all our genericity results), and if its topology is finer than the weak $C^k$-Whitney topology,
i.e., if convergence in $\mathcal E$ implies $C^k$-convergence on compacta.
$C^k$-Whitney type Banach spaces of tensors seem to provide a sufficiently general and adequate
environment in which one can prove genericity results based on Sard--Smale theorem, including a
large variety of situations where one poses asymptotic conditions on the metric tensors.
An argument by Taubes, pointed out to the authors by the referee, allows to extend
all the genericity results presented in this paper to the more elegant context of
the topology of $C^\infty$-convergence on compact subsets. In Section~\ref{sec:smooth}
we will discuss the details of the argument.

In Subsection~\ref{sub:conformal} we study the genericity property of metrics in a given conformal
class. As mentioned above, we restrict ourselves to the case of nondegeneracy of nonlightlike geodesics between fixed
endpoints. In subsection~\ref{sub:splitting} we consider product manifolds $M=M_1\times M_2$, endowed with
metric tensors that make the two factors orthogonal, and we prove a genericity result in this context.
In Subsection~\ref{sub:globhyper} we consider globally hyperbolic Lorentzian metric tensors; by a celebrated
result of Geroch (\cite{Ger}), recently improved by Bernal and S\'anchez (\cite{BerSan2, BerSan3}), these metrics form a subclass
of the family of orthogonally split metric tensors in product manifolds $M_1\times\R$.
Finally, in Section~\ref{sec:stationary}, we will exhibit a counterexample to the transversality condition
in the stationary Lorentzian case.
\end{section}

\begin{section}{Notations and Preliminaries}
\label{sec:preliminaries}
\subsection{Functional analytical preliminaries}
Let $H$ be a Hilbert space with inner product $\langle\cdot,\cdot\rangle$;
given a closed subspace $W\subset H$, we will denote by $P_W:H\to W$ the
orthogonal projection onto $W$.
\begin{lem}\label{thm:lemkercompl}
Let $V$ be a Banach space, $H$ a Hilbert space, $L_1:V\to H$ and $L_2:H\to H$ be bounded linear operators,
with $\mathrm{Im}(L_2)$ closed; set $L=L_1\oplus L_2:V\oplus H\to H$, $L(v,h)=L_1(v)+L_2(h)$, $v\in V$, $h\in H$.
Then, $L$ is surjective if and only if $P_{\mathrm{Im}(L_2)^\perp}\big(\mathrm{Im}(L_1)\big)=\mathrm{Im}(L_2)^\perp$.
If in addition $L_2$ is self-adjoint and $P_{\Ker(L_2)}(\mathrm{Im}(L_1))$ is closed in $\Ker(L_2)$
(this is the case, for instance, if $\Ker(L_2)$ is finite dimensional, i.e., if $L_2$ is Fredholm),
then $L$ is surjective if and only if for all $h\in\Ker(L_2)\setminus\{0\}$ there exists $v\in V$ such that $\langle L_1(v),h\rangle\ne0$.
\end{lem}
\begin{proof}
The first statement is immediate.  If $L_2$ is self-adjoint, then $\mathrm{Im}(L_2)^\perp=\Ker(L_2)$.
Since $P_{\Ker(L_2)}(\mathrm{Im}(L_1))$ is closed,
$P_{\Ker(L_2)}\circ L_1:V\to\Ker(L_2)$ is not surjective
if and only if there exists $h\in\Ker(L_2)$ such that $\langle P_{\Ker(L_2)}(L_1(v)),h\rangle=\langle L_1(v),h\rangle=0$ for
all $v\in V$. The conclusion follows.
\end{proof}

Let us recall that a closed subspace $W$ of a Banach space $V$ is said to be \emph{complemented}
if there exists a closed subspace $W'\subset V$ such that $V=W\oplus W'$; such a space $W'$ will
be called a \emph{complement} of $W$ in $V$.
\begin{lem}\label{thm:lemmascemo}
Let $L:U\to V$ be a linear map between vector spaces, and let $S\subset V$ be a finite codimensional space.
Then, $L^{-1}(S)$ is finite codimensional in $U$, and $\mathrm{codim}_{U}\big(L^{-1}(S)\big)=
\mathrm{codim}_V(S)-\mathrm{codim}_V\big(\mathrm{Im}(L)+S\big)$.
\end{lem}
\begin{proof}
If $\pi:V\to V/S$ is the projection onto the quotient, the linear map $\pi\circ L:U\to V/S$ has kernel $L^{-1}(S)$.
Hence, $\pi\circ L$ defines an injective linear map on the quotient $U/L^{-1}(S)\to V/S$, and so:
\begin{multline*}\mathrm{codim}_V(S)=\mathrm{dim}\big(V/S\big)=\mathrm{dim}\big(U/L^{-1}(S)\big)+
\mathrm{codim}_{V/S}\big(\mathrm{Im}(\pi
\circ L)\big)\\=\mathrm{codim}_U\big(L^{-1}(S)\big)+\mathrm{codim}_V\big(\mathrm{Im}(L)+S\big).\qedhere
\end{multline*}
\end{proof}
\begin{prop}\label{thm:lemsurj}
Let $U$, $V$, $W$ be Banach spaces, $L_1:U\to W$, $L_2:V\to W$ be bounded linear operators, and
assume that $\Ker(L_2)$ is complemented in $V$ (this is the case, for instance, if $V$ is a Hilbert space,
or if $L_2$ is Fredholm) and that $\mathrm{Im}(L_2)$ is finite codimensional in $W$.
Set $L=L_1\oplus L_2:U\oplus V\to W$; then, $\Ker(L)$ is complemented in $U\oplus V$.
\end{prop}
\begin{proof}
Consider the (possibly not closed) subspace $\mathrm{Im}(L_1)\subset W$;  $\mathrm{Im}(L_1)\cap\mathrm{Im}(L_2)$
has finite codimension in $\mathrm{Im}(L_1)$. Namely, if $\pi:W\to W/\mathrm{Im}(L_2)$ is the quotient map,
then the restriction $\pi\vert_{\mathrm{Im}(L_1)}:\mathrm{Im}(L_1)\to W/\mathrm{Im}(L_2)$ has kernel
$\mathrm{Im}(L_1)\cap\mathrm{Im}(L_2)$. Thus, one has an injective linear map from $\mathrm{Im}(L_1)/\big[\mathrm{Im}(L_1)
\cap\mathrm{Im}(L_2)\big]$ to the finite dimensional space $W/\mathrm{Im}(L_2)$, which proves our claim.
Set $\mathrm{Im}(L_1)=\big[\mathrm{Im}(L_1)\cap\mathrm{Im}(L_2)\big]\oplus Z$, with $Z\subset W$ a (closed)
finite dimensional subspace. We now claim that $\Ker(L_1)$ has finite codimension in $L_1^{-1}(Z)$; namely,
one has an injective linear map from $L_1^{-1}(Z)/\Ker(L_1)$ to $Z$. Set $L_1^{-1}(Z)=\Ker(L_1)\oplus U'$,
with $U'$ a (closed) finite dimensional subspace of $U$. Finally, let $V'$ be a complement of
$\Ker(L_2)$ in $V$; we will now show that $U'\oplus V'$ is a complement of $\Ker(L)$ in $U\oplus V$.
Assume $(x,y)\in U'\oplus V'$ with $L_1(x)+L_2(y)=0$; since $U'\subset L_1^{-1}(Z)$, then $L_1(x)\in Z$.
But $L_1(x)=-L_2(y)\in\mathrm{Im}(L_2)$, thus $L_1(x)\in Z\cap\big(\mathrm{Im}(L_1)\cap\mathrm{Im}(L_2)\big)=\{0\}$,
i.e., $L_1(x)=L_2(y)=0$. Thus, $x\in U'\cap\Ker(L_1)=\{0\}$ and $y\in V'\cap\Ker(L_2)=\{0\}$, which proves
that $\big[U'\oplus V'\big]\cap\Ker(L)=\{0\}$.

 Let now $(x,y)\in U\oplus V$ be arbitrary; write $L_1(x)=L_1(u)+z$, where $u\in U$, $L_1(u)\in\mathrm{Im}(L_2)$
 and $z\in Z$. Since $z\in Z\subset\mathrm{Im}(L_1)$, one has $z=L_1(a)$ for some $a\in U'$; thus, $x=u+a+b$ for some
 $b\in\Ker(L_1)$. Choose $w\in V'$ such that $L_1(u)=L_2(w)$, and set $y=c+v$, where $c\in\Ker(L_2)$ and
 $v\in V'$. Then, $(u+b,c-w)\in\Ker(L)$, $(a,v+w)\in U'\oplus V'$ and $(x,y)=(u+b,c-w)+(a,v+w)$, which
 proves that $\Ker(L)+\big[U'\oplus V']=U\oplus V$.
\end{proof}
\subsection{Geometric preliminaries}\label{sub:geompreliminaries}
Let $M$ be a smooth manifold with $\Dim(M)\ge2$ and let $\nabla$ be an arbitrarily fixed symmetric connection on $TM$.
Given another (symmetric) connection $\nabla'$ on $TM$, there exists a (symmetric) $(1,2)$-tensor
$\Gamma$ on $M$ defined by:
\[\nabla'=\nabla+\Gamma,\]
that will be called the \emph{Christoffel tensor of $\nabla'$ relatively to $\nabla$}.
If $\nabla^g$ is the Levi--Civita connection of some semi-Riemannian
metric tensor $g$ on $M$, then using Koszul's formula, its Christoffel tensor relative to $\nabla$ is computed as follows:
\begin{equation}\label{eq:Christoffeltensor}
g\big(\Gamma^g(X,Y),Z\big)=\frac12\big[\nabla g(X,Z,Y)+\nabla g(Y,Z,X)-\nabla g(Z,X,Y)\big].
\end{equation}
For all $x\in M$ and all $v\in T_xM$, we will denote by $\Gamma^g_x(v):T_xM\to T_xM$ the map defined by $\Gamma^g_x(v)w=\Gamma^g_x(v,w)$,
for all $w\in T_xM$, and by $\Gamma_x^g(v)^*:T_xM^*\to T_xM^*$ its adjoint.
The curvature tensor $R^g$ of the connection $\nabla^g$ will be chosen with the following sign convention:
$R^g(X,Y)=[\nabla^g_X,\nabla^g_Y]-\nabla^g_{[X,Y]}$. The symbol $\exp$ will denote the exponential map
of the connection $\nabla$.

Given a smooth vector bundle $\pi:E\to M$ over $M$, we will denote by $\mathbf\Gamma(E)$ the space
of all smooth sections of $E$; given a smooth map between manifolds $f:N\to M$, then $f^*(E)$ will denote the
pull-back bundle over $N$. The fiber $\pi^{-1}(x)$ over a point $x\in M$ will be denoted by $E_x$; the dimension
of the typical fiber of $E$ will be called the \emph{rank} of $E$.
In this paper, we will be mostly interested in \emph{tensor bundles} over $M$, i.e., all those vector bundles
obtained by functorial constructions from the tangent bundle $TM$ and the cotangent bundle $TM^*$.
Given nonnegative integers $r,s$, we will denote by ${TM^*}^{(r)}\otimes TM^{(s)}$ the tensor product
of $r$ copies of $TM^*$ and $s$ copies of $TM$; sections of ${TM^*}^{(r)}\otimes TM^{(s)}$ are
called \emph{tensors of type $(s,r)$} on $M$.

The following is a result that says that we can find global  sections of a vector bundle with prescribed
value and covariant derivative along a sufficiently short curve in $M$.
\begin{lem}\label{thm:constrloc-strong}
Let $\pi:E\to M$ be a smooth vector bundle endowed with a connection $\nabla$, let $\gamma:[a,b]\to M$ be a smooth immersion
and let $V\in\Gamma\big(\gamma^*(TM)\big)$ be a smooth vector field along $\gamma$ such that $V_{t_0}$ is not parallel
to $\dot\gamma(t_0)$ for some $t_0\in\left]a,b\right[$. Then, there exists an open interval $I\subset[a,b]$ containing $t_0$
with the property that, given smooth sections $H$ and $K$ of $\gamma^*(E)$ with compact support in $I$
and given any open set $U$ containing $\gamma(I)$,
then there exists $h\in\mathbf\Gamma(E)$ with compact support contained in $U$,
such that $h_{\gamma(t)}=H_t$ and $\nabla_{V_t}h=K_t$ for all $t\in I$.
\end{lem}
\begin{proof}
Let $I\subset\left]a,b\right[$ be a sufficiently small open interval such that $\gamma\vert_I$ is an embedding
and such that $V_t$ is not parallel to $\dot\gamma(t)$ for all $t\in I$;
let $S\subset M$ be a smooth hypersurface containing $\gamma(I)$ and such that $V_t\not\in T_{\gamma(t)}S$ for all $t\in I$.
Choose a smooth section $V\in\mathbf\Gamma\big(S^*(TM)\big)$ such that $V\big(\gamma(t)\big)=V_t$ for all $t\in I$.
By possibly reducing the size of $I$ and $S$, we can assume the existence of a small positive number $\varepsilon$ and
a diffeomorphism $\phi:S\times\left]-\varepsilon,\varepsilon\right[\ni(x,\lambda)\mapsto\phi(x,\lambda)\in \widetilde U\subset U$, where $\widetilde U$
is an open subset of $M$ contained in $U$ that contains $\gamma(I)$, such that
$\frac{\partial\phi}{\partial\lambda}(x,0)=V(x)$ for all
$x\in S$. For instance, such a diffeomorphism can be constructed using the exponential
map $\exp'$ of some connection $\nabla'$ in $TM$ by setting $\phi(x,\lambda)=\exp'_x\big(\lambda V(x)\big)$ for all
$(x,\lambda)\in S\times\left]-\varepsilon,\varepsilon\right[$. Clearly, $\widetilde U$ can be chosen small enough so that $E\vert_{\widetilde U}$ admits
a trivialization; let $r\in\N$ be the rank of $E$ and let $p(x,\lambda):\R^r\to E_{\phi(x,\lambda)}$
be a smooth referential of
$\phi^*\big(E\vert_{\widetilde U}\big)$ with the property that $\frac{\mathrm D}{\mathrm d\lambda}p(x,\lambda)=0$, i.e., $p$ is parallel along the curves $\left]-\varepsilon,\varepsilon\right[\ni
\lambda\mapsto\phi(x,\lambda)$.
For instance, such referential $p$ can be chosen by selecting an arbitrary smooth referential of $E$ along $S$, and
then extending by parallel transport along the curve $\lambda\mapsto\phi(x,\lambda)$. The problem of determining the required section $h$ is now reduced to the
search of a smooth map $\widetilde h:S\times\left]-\varepsilon,\varepsilon\right[\to\R^r$ having compact support such that:
\begin{itemize}
\item $\widetilde h\big(\gamma(t),0\big)=p\big(\gamma(t),0\big)^{-1}H_t$;
\item $\dfrac{\partial\widetilde h}{\partial\lambda}\big(\gamma(t),0\big)=p\big(\gamma(t),0\big)^{-1}K_t$,
\end{itemize}
for all $t\in I$. Once such $\widetilde h$ has been determined, the desired section $h$ will be obtained by setting
$h\big(\phi(x,\lambda)\big)=p(x,\lambda)\circ\widetilde h(x,\lambda)$
for all $(x,\lambda)\in S\times\left]-\varepsilon,\varepsilon\right[$ and $h=0$ outside $\widetilde U$.

The function $\widetilde h$ can be constructed as follows. First, choose smooth maps $\widetilde H,\widetilde K:S\to\R^r$ having compact support
such that $p\big(\gamma(t),0\big)\circ\widetilde H\big(\gamma(t)\big)=H_t$ and $p\big(\gamma(t),0\big)\circ\widetilde K\big(\gamma(t)\big)=K_t$
for all $t\in I$. Finally, define $\widetilde h(x,\lambda)=\widetilde H(x)+f(\lambda)\widetilde K(x)$, where
$f:\left]-\varepsilon,\varepsilon\right[\to\R$ is a smooth function with compact support such that $f(\lambda)=\lambda$ near $\lambda=0$.
This concludes the construction and proves the Lemma.
\end{proof}

Given a $g$-geodesic $\gamma:I\to M$ in $M$, a \emph{Jacobi field} along $\gamma$ is a smooth vector
field $J$ along $\gamma$ that satisfies the second order linear equation $(\mathbf D^g)^2J(t)=R^g\big(\dot\gamma(t),J(t)\big)\,\dot\gamma(t)$
for all $t$, where $\mathbf D^g$ denotes covariant differentiation along $\gamma$ relatively to the connection
$\nabla^g$. The endpoints of $\gamma$ are said to be \emph{conjugate} along $\gamma$ if there exists a non trivial
Jacobi field along $\gamma$ that vanishes at both endpoints of $I$.
Affine multiples of the tangent field $\dot\gamma$ are Jacobi fields; conversely, the only Jacobi fields along
$\gamma$ that are everywhere parallel to $\dot\gamma$ must be affine multiples of $\dot\gamma$. Other than
that, Jacobi fields are parallel to the tangent field $\dot\gamma$ only at isolated points:
\begin{lem}\label{thm:Jnontparalleltogamma}
Let $\gamma:[a,b]\to M$ be a geodesic in $(M,g)$, and let $J$ be a Jacobi field
which is not everywhere parallel to $\dot\gamma$. Then, the set:
\[\big\{t\in[a,b]:J_t\ \text{is parallel to}\ \dot\gamma(t)\big\}\]
is finite.
\end{lem}
\begin{proof}
Since $\dot\gamma$ is parallel, the covariant differentiation operator $\mathbf D^g$ defines a connection
on the quotient bundle $\bigcup_{t\in[a,b]}T_{\gamma(t)}M/\R\dot\gamma(t)$ over the interval $[a,b]$, that
will be denoted by $\widetilde{\mathbf D}$.
Moreover, by the anti-symmetry of the curvature tensor, the linear operator $R^g\big(\dot\gamma(t),\cdot)\dot\gamma(t):T_{\gamma(t)}M\to
T_{\gamma(t)}M$ passes to the quotient and gives a well defined operator
$\widetilde R_t:T_{\gamma(t)}M/\R\dot\gamma(t)\to T_{\gamma(t)}M/\R\dot\gamma(t)$.
Thus, the class $\widetilde J=J+\R\dot\gamma$ satisfies the second order linear differential equation
$\widetilde{\mathbf D}^2\widetilde J=\widetilde R\widetilde J$. If the zeroes of $\widetilde J$ were not isolated,
then $\widetilde J$ would be identically zero, i.e., $J$ would be everywhere parallel to $\dot\gamma$.
\end{proof}

\end{section}

\begin{section}{An abstract genericity result}
\label{sec:abstrgenresult}
In this section we will study the nondegeneracy of critical points of a smoothly varying family of
variational problems; we will prove the result of \cite[Theorem 1.2]{Whi} in the
context of Banach and Hilbert manifolds. The approach followed is classical (see \cite[Chapter~4]{AbrRob},
or \cite[Section~2.11]{AbboMaj3}), and several of the results presented in this section
are very likely already existing in the literature in some form. The authors have found White's
formulation of the transversality assumption (see \eqref{thm:transversality}) particularly
well suited for their purposes, and decided to write complete proofs of its extension
to the Banach manifold setting.

Recall that, given Banach manifolds $\mathcal X$ and $\mathcal Y$, a smooth submanifold $\mathcal Z\subset\mathcal Y$, and
a $C^1$-map $F:\mathcal X\to\mathcal Y$, then $F$ is said to be \emph{transversal} to $\mathcal Z$ if for all $\mathfrak x_0\in F^{-1}(\mathcal Z)$,
$\mathrm dF(\mathfrak x_0)^{-1}\big(T_{F(\mathfrak x_0)}\mathcal Z\big)$ is complemented in $T_{\mathfrak x_0}\mathcal X$ and $\mathrm{Im}\big(\mathrm{d}F(\mathfrak x_0)\big)+T_{F(\mathfrak x_0)}\mathcal Z=T_{F(\mathfrak x_0)}\mathcal Y$.
Under these circumstances, $\mathcal M=F^{-1}(\mathcal Z)$ is a smooth embedded submanifold of $\mathcal X$, and for all $\mathfrak x_0\in\mathcal M$,
$T_{\mathfrak x_0}\mathcal M$ is given by $\mathrm dF(\mathfrak x_0)^{-1}\big(T_{F(\mathfrak x_0)}\mathcal Z\big)$.
\begin{prop}\label{thm:transversality}
Let $X$ be a  Banach manifold, $Y$ a  Hilbert manifold, and let $\mathcal A\subset X\times Y$ be an open subset.
Assume that $f:\mathcal A\to\R$ is a map of class $C^k$, with $k\ge2$, and with the property that for every $(x_0,y_0)\in\mathcal A$ such that
$\frac{\partial f}{\partial y}(x_0,y_0)=0$, the Hessian \[\frac{\partial^2f}{\partial y^2}(x_0,y_0):T_{y_0}Y\longrightarrow T_{y_0}Y^*\cong T_{y_0}Y\]
has finite codimensional image (i.e., $\frac{\partial^2f}{\partial y^2}(x_0,y_0)$ is a Fredholm operator\footnote{%
Recall that the image of a bounded linear operator, if finite codimensional, is automatically closed.}).\par Then, the map $\frac{\partial f}{\partial y}:\mathcal A\to TY^*$ is transversal to the zero section of $TY^*$ if and only
if for all $(x_0,y_0)$ with $\frac{\partial f}{\partial y}(x_0,y_0)=0$ and all
$w\in\Ker\Big[\frac{\partial^2f}{\partial y^2}(x_0,y_0)\Big]\setminus\{0\}$
there exists $v\in T_{x_0}X$ such that
\[
\frac{\partial^2f}{\partial x\partial y}(x_0,y_0)(v,w)\ne0,
\]
i.e.,
\begin{equation}\label{eq:transvcond}
\Ker\left(\frac{\partial^2f}{\partial y^2}(x_0,y_0)\right)\bigcap\mathrm{Im}\left(\frac{\partial^2f}{\partial x\partial y}(x_0,y_0)\right)^\perp=\{0\}.
\end{equation}
\end{prop}
\begin{rem}
Observe that, given $y_0$, the map $x\mapsto\frac{\partial f}{\partial y}(x,y_0)$ takes values in the fixed Hilbert space
$T_{y_0}Y^*$, so that the second derivative $\frac{\partial^2f}{\partial x\partial y}(x_0,y_0)$ is well
defined without the use of a connection on $TY^*$. Similarly, the second derivative $\frac{\partial^2f}{\partial y^2}(x_0,y_0)$ is well
defined when $\frac{\partial f}{\partial y}(x_0,y_0)=0$, and it is the Hessian of the function $y\mapsto f(x_0,y)$ at the
critical point $y_0$.
\end{rem}
\begin{proof}
Denote by $\mathbf 0$ the zero section of $TY^*$. For all $y\in Y$, denoting by $0_y$ the zero in $T_yY^*$,
the tangent space $T_{0_y}\mathbf 0$ is identified canonically with $T_yY$, so that $T_{0_y}TY^*\cong T_yY\oplus T_yY^*$; let $\pi_y:T_{0_y}TY^*\to T_yY^*$
denote the projection relative to this decomposition.
Given $(x_0,y_0)\in\mathcal A$ with $\frac{\partial f}{\partial y}(x_0,y_0)=0$, the composition \[\pi_{y_0}\circ\mathrm d\left(\frac{\partial f}{\partial y}\right)(x_0,y_0):T_{x_0}X\oplus T_{y_0}Y
\longrightarrow T_{y_0}Y^*\]
is given by the direct sum of the bounded operators:
\[L_1:=\frac{\partial^2 f}{\partial x\partial y}(x_0,y_0):T_{x_0}X\longrightarrow  T_{y_0}Y^*\cong T_{y_0}Y\]
and
\[L_2:=\frac{\partial^2f}{\partial y^2}(x_0,y_0):T_{y_0}Y\longrightarrow T_{y_0}Y^*\cong T_{y_0}Y.\]
Transversality of $\frac{\partial f}{\partial y}$ to the zero section of $TY^*$ is equivalent to  $\Ker(L_1\oplus L_2)$ being complemented in $T_{x_0}X\oplus T_{y_0}Y$
and  $L_1\oplus L_2$ being  surjective.  The condition that $\Ker(L_1\oplus L_2)$ is complemented in $T_{x_0}X\oplus T_{y_0}Y$
follows immediately from Proposition~\ref{thm:lemsurj}, which uses our assumptions
on the Hessian $\frac{\partial^2f}{\partial y^2}(x_0,y_0)$. By Lemma~\ref{thm:lemkercompl},
using the fact that $L_2$ is self-adjoint, the surjectivity of $L_1\oplus L_2$
 is equivalent to our
final assumption on the mixed second derivative $\frac{\partial^2f}{\partial x\partial y}(x_0,y_0)$.
This concludes the proof.
\end{proof}
\begin{cor}\label{thm:Fredholmamap}
In the hypotheses of Proposition~\ref{thm:transversality}, assume that the transversality condition \eqref{eq:transvcond}
is satisfied at every point $(x_0,y_0)$ with $\frac{\partial f}{\partial y}(x_0,y_0)=0$. Then, the set:
\[\mathfrak M=\Big\{(x,y)\in\mathcal A:\frac{\partial f}{\partial y}(x,y)=0\Big\}\]
is an embedded $C^{k-1}$-submanifold of $X\times Y$. For $(x_0,y_0)\in\mathfrak M$, the tangent space $T_{(x_0,y_0)}\mathfrak M$
is given by:
\begin{equation}\label{eq:spaziotangM}
T_{(x_0,y_0)}\mathfrak M=\Big\{(v,w)\in T_{x_0}X\oplus T_{y_0}Y:\frac{\partial^2f}{\partial x\partial y}(x_0,y_0)v+\frac{\partial^2f}{\partial y^2}(x_0,y_0)w=0\Big\}.\qed
\end{equation}
\end{cor}
Let us recall that a Morse function on a Hilbert manifold is a smooth map all of whose critical points are (strongly) nondegenerate.
A subset of a metric space is said to be \emph{generic} if it is the countable intersection of dense open subsets; by Baire's theorem,
a generic set is dense.
\begin{cor}\label{thm:genericnondegeneracy}
Under the assumptions of Corollary~\ref{thm:Fredholmamap},
if $\Pi:X\times Y\to X$ is the projection onto the first factor, then the restriction of $\Pi$ to $\mathfrak M$ is a nonlinear $C^{k-1}$ Fredholm
map of index zero.
The critical points of $\Pi\vert_{\mathfrak M}$ are elements $(x_0,y_0)\in\mathfrak M$ such that $y_0$ is a degenerate critical point of the functional
$\mathcal A_{x_0}\ni y\mapsto f(x_0,y)\in\R$, where \[\mathcal A_x=\big\{y\in Y:(x,y)\in\mathcal A\big\}.\]
If $X$ and $Y$ are separable, then the set of $x\in X$ such that the functional $\mathcal A_x\ni y\mapsto f(x,y)\in\R$
is a Morse function is generic in the open set $\Pi(\mathcal A)\subset X$.
\end{cor}
\begin{proof}
Fix $(x_0,y_0)\in\mathfrak M$. The kernel of $\mathrm d\Pi(x_0,y_0)\vert_{T_{(x_0,y_0)}\mathfrak M}$
is given by $T_{(x_0,y_0)}\mathfrak M\cap\big[\{0\}\oplus T_{y_0}Y\big]$. This space is (isomorphic to)
$\Ker\big(\frac{\partial^2f}{\partial y^2}(x_0,y_0)\big)$, which is finite dimensional.
From~\eqref{eq:spaziotangM}, the image $\mathrm d\Pi(x_0,y_0)\left(T_{(x_0,y_0)}\mathfrak M\right)$ is given by the inverse image
\[\left[\frac{\partial^2 f}{\partial x\partial y}(x_0,y_0)\right]^{-1}\Big[\mathrm{Im}
\Big(\frac{\partial^2f}{\partial y^2}(x_0,y_0)\Big)\Big];\]
since $\frac{\partial^2f}{\partial y^2}(x_0,y_0)$ is Fredholm, its image has finite
codimension in $T_{y_0}Y$. By Lemma~\ref{thm:lemmascemo}, also
$\left(T_{(x_0,y_0)}\mathfrak M\right)$ has finite codimension in $T_{x_0}X$, so that  $\mathrm d\Pi(x_0,y_0)\left(T_{(x_0,y_0)}\mathfrak M\right)$ is closed and therefore Fredholm. In fact, since by assumption~\eqref{eq:transvcond}
the linear map $\frac{\partial^2 f}{\partial x\partial y}(x_0,y_0)
\oplus\frac{\partial^2f}{\partial y^2}(x_0,y_0)$ is surjective, we have that by Lemma~\ref{thm:lemmascemo}
the codimension of the image of $\mathrm d\Pi(x_0,y_0)\vert_{T_{(x_0,y_0)}\mathfrak M}$ equals the codimension of
$\mathrm{Im}\big(\frac{\partial^2f}{\partial y^2}(x_0,y_0)\big)$; as $\frac{\partial^2f}{\partial y^2}(x_0,y_0)$ is a self-adjoint Fredholm operator, this codimension coincides with the dimension of \[\mathrm{Im}\left(\frac{\partial^2f}{\partial^2 y}(x_0,y_0)\right)^\perp=
\mathrm{Ker}\left(\frac{\partial^2f}{\partial^2 y}(x_0,y_0)\right),\] so that
the Fredholm index of $\mathrm d\Pi(x_0,y_0)\vert_{T_{(x_0,y_0)}\mathfrak M}$ is equal to zero.

It is easily seen that $(x_0,y_0)$ is a regular point of $\Pi\vert_{\mathfrak M}$, i.e.,
that $\mathrm d\Pi(x_0,y_0)\vert_{T_{(x_0,y_0)}\mathfrak M}$ is surjective, if and only if:
\begin{equation}\label{eq:dPisurjective}
\mathrm{Im}\left(\frac{\partial^2f}{\partial x\partial y}(x_0,y_0)\right)\subset\mathrm{Im}\left(\frac{\partial^2f}{\partial y^2}(x_0,y_0)\right);
\end{equation}
using again that
$\mathrm{Im}\left(\frac{\partial^2f}{\partial^2 y}(x_0,y_0)\right)^\perp=
\mathrm{Ker}\left(\frac{\partial^2f}{\partial^2 y}(x_0,y_0)\right)$,
and taking orthogonal complements,
\eqref{eq:dPisurjective} becomes:
\[\mathrm{Im}\left(\frac{\partial^2f}{\partial x\partial y}(x_0,y_0)\right)^\perp\supset\Ker\left(\frac{\partial^2f}{\partial y^2}(x_0,y_0)\right).\]
Using assumption~\eqref{eq:transvcond}, $(x_0,y_0)$ is a regular point of $\Pi\vert_{\mathfrak M}$
if and only if $\Ker\left(\frac{\partial^2f}{\partial y^2}(x_0,y_0)\right)$ is trivial,
i.e., if and only if $x_0$ is a nondegenerate critical point of $x\mapsto f(x,y_0)$.

Thus, the set of $x\in X$ such that the functional $\mathcal A_x\ni y\mapsto f(x,y)\in\R$ is a Morse function
coincides with the set of regular values of the map $\Pi\vert_{\mathfrak M}$.
The last statement follows now immediately from Corollary~\ref{thm:Fredholmamap} and Sard--Smale's theorem (see \cite{Sma}).
\end{proof}
\begin{rem}\label{thm:remordinederivatemiste}
We will apply Corollary~\ref{thm:genericnondegeneracy} in situations where the Banach manifold $X$ is indeed
an open subset of a Banach space $E$. In this case, the partial derivative $\frac{\partial f}{\partial x}$ is a map
on $X\times Y$ taking value in the fixed Banach space $E^*$, and thus it can be differentiated with respect
to the second variable $y$. Given $(x_0,y_0)\in\mathfrak M$, we have two maps:
\[\frac{\partial^2 f}{\partial x\partial y}(x_0,y_0):E\to T_{y_0}Y^*,\quad\text{and}\quad \frac{\partial^2 f}{\partial y\partial x}(x_0,y_0):
T_{y_0}Y\to E^*.\]
Using local charts and Schwarz Lemma, it is easy to see that these two maps are transpose of each other.
In particular, if we consider  $\frac{\partial^2 f}{\partial x\partial y}(x_0,y_0)$ as a bilinear form on
$E\times T_{y_0}Y$ and $\frac{\partial^2 f}{\partial y\partial x}(x_0,y_0)$ as a bilinear form on $T_{y_0}Y\times E$,
then:
\[\frac{\partial^2 f}{\partial x\partial y}(x_0,y_0)[v,w]=\frac{\partial^2 f}{\partial y\partial x}(x_0,y_0)[w,v],
\quad\forall\,v\in E,\,w\in T_{y_0}Y.\]
\end{rem}
\end{section}

\begin{section}{Morse geodesic functionals}
\label{sec:morsegeofunctionals}
\subsection{Semi-Riemannian metrics}\label{sub:semiRiemannianmetrics}
\label{sub:metrics}
Let us consider a smooth\footnote{%
For the remainder of the article, we will be somewhat sloppy about the use of the adjective ``smooth''.
In the case of manifolds, by smooth we will always mean ``of class $C^k$, with $k\ge3$'', and
in the case of tensors, in particular metric tensors, smooth will mean ``of class $C^k$, with $k\ge2$''. This guarantees
that the corresponding geodesic action functionals are of class at least $C^2$.
Clearly, manifolds are to be of class strictly larger than the required regularity class of tensors.}
manifold $M$ with $\Dim(M)=n$. Given $k\ge2$ and $\nu\in\{0,\ldots,n\}$, we will
denote by $\mathrm{Met}^k_\nu(M)$ the set of all metric tensors $g$ on $M$ of class $C^k$ and having index $\nu$.
This is a subset of the vector space $\mathbf\Gamma_{\mathrm{sym}}^k(TM^*\otimes TM^*)$ of all sections $\mathfrak b$ of
class $C^k$ of the vector bundle $TM^*\otimes TM^*$ such that $\mathfrak b_x:T_xM\times T_xM\to\R$ is symmetric
for all $x$.

It will be interesting to consider the case of \emph{non compact} manifolds $M$, in which case there is no canonical
Banach structure on the space of tensors over $M$.
In order to overcome this problem, it will be useful to consider the following definition.
A vector subspace $\mathcal E$ of  $\mathbf\Gamma_{\mathrm{sym}}^k(TM^*\otimes TM^*)$ will be called\footnote{%
In this paper we will only be interested in metric tensor fields, but clearly a similar definition may be given
for tensor fields of all kind over $M$.}
a \emph{$C^k$-Whitney type Banach space of tensor fields} over $M$ when:
\begin{itemize}
\item[(a)] $\mathcal E$ contains all tensor fields in $\mathbf\Gamma_{\mathrm{sym}}^k(TM^*\otimes TM^*)$ having
compact support;
\item[(b)] $\mathcal E$ is endowed with a Banach space norm $\Vert\cdot\Vert_{\mathcal E}$ with the property that
$\Vert\cdot\Vert_{\mathcal E}$-convergence of a sequence implies convergence in the weak Whitney $C^k$-topology.
\end{itemize}
More explicitly, axiom~(b) above means that given any sequence $(\mathfrak b_n)_{n\in\N}$ and an
element $\mathfrak b_\infty$ in $\mathcal E$
such that $\lim\limits_{n\to\infty}\Vert\mathfrak b_n-\mathfrak b_\infty\Vert_{\mathcal E}=0$, and given any compact
subset $K\subset M$, then the restriction $\mathfrak b_n\vert_{K}$ tends to $\mathfrak b_\infty\vert_K$ in
the $C^k$-topology as $n\to\infty$.

\begin{example}
\label{exa:exaWtBstf}
Examples of $C^k$-Whitney type Banach spaces of tensor fields over $M$ can be obtained
easily introducing an auxiliary Riemannian metric $g_{\mathrm R}$ on $M$, whose Levi--Civita connection will be denoted
 by $\overline\nabla$. The choice of the Riemannian metric $g_{\mathrm R}$ induces naturally a connection on all vector bundles
over $M$ that are obtained by functorial constructions from the tangent bundle $TM$. Moreover, for all $r,s\in\N$,
we have Hilbert space norms on every tensor product ${T_xM^*}^{(r)}\otimes {T_xM}^{(s)}$ induced by $g_{\mathrm R}$;
all these norms will be denoted by the same symbol $\Vert\cdot\Vert_{\mathrm R}$.
Then, we will denote by $\mathbf\Gamma_{\mathrm{sym}}^k(TM^*\otimes TM^*;g_{\mathrm R})$
the subset of $\mathbf\Gamma_{\mathrm{sym}}^k(TM^*\otimes TM^*)$ consisting of all section
$\mathfrak b$ such that:
\begin{equation}\label{eq:defnormak}
\Vert\mathfrak b\Vert_k=\max_{j=0,\ldots,k}\left[\sup_{x\in M}\left\Vert\overline\nabla^j\mathfrak b(x)\right\Vert_{\mathrm R}\right]<+\infty.
\end{equation}
When $M$ is compact $\mathbf\Gamma_{\mathrm{sym}}^k(TM^*\otimes TM^*;g_{\mathrm R})=\mathbf\Gamma_{\mathrm{sym}}^k(TM^*\otimes TM^*)$.
Endowed with the norm $\Vert\cdot\Vert_k$ in \eqref{eq:defnormak}, $\mathbf\Gamma_{\mathrm{sym}}^k(TM^*\otimes TM^*;g_{\mathrm R})$
is a separable normed space, which is complete provided that the Riemannian metric $g_{\mathrm R}$ is chosen to be complete.
Clearly, $\mathbf\Gamma_{\mathrm{sym}}^k(TM^*\otimes TM^*;g_{\mathrm R})$ contains all elements in $\mathbf\Gamma_{\mathrm{sym}}^k(TM^*\otimes TM^*)$
having compact support. Moreover, $\Vert\cdot\Vert_k$-convergence implies $C^k$-convergence on compact sets.
Thus, $\mathbf\Gamma_{\mathrm{sym}}^k(TM^*\otimes TM^*;g_{\mathrm R})$ is an example of $C^k$-Whitney type Banach spaces of tensor fields over $M$.
\end{example}
Other examples of $C^k$-Whitney type Banach spaces of tensor fields over $M$ can be obtained by considering
elements in $\mathbf\Gamma_{\mathrm{sym}}^k(TM^*\otimes TM^*)$ satisfying suitable boundedness assumptions
at infinity on the first $k$ covariant derivatives. Asymptotic flatness is a typical assumption, particularly fashionable among physicists.

In the statements of some of our results, we will consider open subsets $\mathcal A$ of
a given $C^k$-Whitney type Banach space $\mathcal E$ of tensor fields over $M$, where the elements of
$\mathcal A$ are assumed to be semi-Riemannian metric tensors of a given index.
It is easy to show that, when $M$ is not compact, the set $\mathrm{Met}^k_\nu(M)\cap\mathbf\Gamma_{\mathrm{sym}}^k(TM^*\otimes TM^*;g_{\mathrm R})$
is \emph{not} open in $\mathbf\Gamma_{\mathrm{sym}}^k(TM^*\otimes TM^*;g_{\mathrm R})$. A typical
open\footnote{%
In order to see that the set $\mathrm{Met}^k_{\nu,\star}(M;g_{\mathrm R})$ is open in
$\mathbf\Gamma_{\mathrm{sym}}^k(TM^*\otimes TM^*;g_{\mathrm R})$, one uses the fact that the
function $A\mapsto\lambda_*(A)=\min\big\{\vert\lambda\vert:\text{$\lambda$ is an eigenvalue of $A$}\big\}$
is Lipschitz continuous on the set of symmetric operators $A$ on $\R^n$. This is proved easily using
the equality $\lambda_*(A)=\min_{\Vert x\Vert=1}\Vert Ax\Vert$, from which one deduces
that $\left\vert\lambda_*(A)-\lambda_*(B)\right\vert\le\Vert A-B\Vert$ for all symmetric
operators $A$ and $B$.}
 subset of $\mathbf\Gamma_{\mathrm{sym}}^k(TM^*\otimes TM^*;g_{\mathrm R})$ consisting of
semi-Riemannian metric tensors of index $\nu$ is:
\[\mathrm{Met}^k_{\nu,\star}(M;g_{\mathrm R})=\Big\{\mathfrak b\in\mathrm{Met}^k_\nu(M)
\cap\mathbf\Gamma_{\mathrm{sym}}^k(TM^*\otimes TM^*;g_{\mathrm R}):\sup_{x\in M}\Vert\mathfrak b_x^{-1}\Vert_{\mathrm R}<+\infty \Big\};\]
here, $\mathfrak b_x^{-1}$ is the inverse of $\mathfrak b_x$ seen as a linear operator
$\mathfrak b_x:T_xM\to T_xM^*$. The assumption $\sup\limits_{x\in M}\Vert\mathfrak b_x^{-1}\Vert_{\mathrm R}<+\infty$
is equivalent to requiring that the eigenvalue with minimum absolute value of the $g_{\mathrm R}$-symmetric operator
$\mathfrak b_x$ stays away from $0$ uniformly on $M$.

\noindent
Let $p,q\in M$ be fixed points, and let $\Omega_{p,q}(M)$ denote the set of all curves $\gamma:[0,1]\to M$ of Sobolev
class $H^1$ such that $\gamma(0)=p$ and $\gamma(1)=q$; it is well known that $\Omega_{p,q}(M)$ is endowed with a Hilbert
manifold structure modeled on the separable Hilbert space $H^1_0\big([0,1],\R^n\big)$.
For $\gamma\in\Omega_{p,q}(M)$, the pull-back bundle $\gamma^*(TM)$ is endowed with a
Riemannian structure on the fibers induced by the Riemannian structure $g_{\mathrm R}$.
The tangent space $T_\gamma\Omega_{p,q}(M)$ is identified with the Hilbertable space
of all sections $V$ of $\gamma^*(TM)$ having Sobolev class $H^1$, and satisfying $V(0)=V(1)=0$.
For the purposes of this paper,  the choice of a specific Hilbert--Riemann structure on
the infinite dimensional manifold $\Omega_{p,q}(M)$ will not be relevant; however, it will
be useful to have at disposal the following inner product on the tangent spaces $T_\gamma\Omega_{p,q}(M)$:
\begin{equation}\label{eq:definnerprTgamma}
\llangle V,W\rrangle=\int_0^1g_{\mathrm R}(\mathbf D^{\mathrm R}V,\mathbf D^{\mathrm R}W)\,\mathrm dt.
\end{equation}
Here, $g_{\mathrm R}$ is an arbitrarily fixed complete Riemannian metric on $M$ and $\mathbf D^{\mathrm R}$
denotes covariant differentiation of vector fields along $\gamma$ with respect to the Levi--Civita connection
of  $g_{\mathrm R}$.

\subsection{Genericity of metrics without degenerate geodesics}
\label{sub:semiRiemgeneric}
We will henceforth consider a fixed $C^k$-Whitney type Banach space $\mathcal E$ of tensor fields over $M$
and a (non empty) open subset $\mathcal A$ of $\mathcal E$ with $\mathcal A\subset\mathcal E\cap\mathrm{Met}^k_\nu(M)$.
A complete Riemannian metric $g_{\mathrm R}$ is also assumed to be fixed, in order to use the
Hilbert manifold structure \eqref{eq:definnerprTgamma} in $\Omega_{p,q}(M)$.
Consider the geodesic action functional:
\[F:\mathcal A\times\Omega_{p,q}(M)\longrightarrow\R\]
defined by:
\[F(g,\gamma)=\tfrac12\int_0^1g(\dot\gamma,\dot\gamma)\,\mathrm dt.\]
This is a map of class $C^k$. More precisely, $F$ is smooth (i.e., $C^\infty$) in the variable $g\in\mathcal A$, while in the
variable $\gamma$ it is of class $C^k$, the same regularity required for the metrics.
This is easily proved, observing that taking $j$ derivatives of $F$ with respect to the variable $\gamma$ involves\footnote{%
For instance, the first derivative $\frac{\partial F}{\partial\gamma}(\gamma_0,g_0)$ in the direction $V\in T_\gamma\Omega_{p,q}(M)$
is given by the integral $\int_0^1g_0(\dot\gamma_0,\mathbf D^{g_0}V)\,\mathrm dt$, where $\mathrm D^{g_0}$ is the covariant derivative of
vector fields along $\gamma$ relatively to the Levi--Civita connection $\nabla^{g_0}$ of $g_0$. This requires the Christoffel tensors of $g$, which are computed
in terms of the first derivatives of the metric coefficients. The second derivative $\frac{\partial^2 F}{\partial\gamma^2}(\gamma_0,g_0)$
involves the curvature tensor of $\nabla^{g_0}$ (see formula \eqref{eq:dersecrespgamma}), i.e., the second derivative of $g$.
Higher order derivatives of $F$ with respect to $\gamma$ are computed in terms of higher order covariant derivatives of the curvature tensor
of $\nabla^{g_0}$.}
the first $j$ derivatives of the metric $g$.

Given $g_0\in\mathcal A$ and $\gamma_0\in\Omega_{p,q}(M)$, then
$\frac{\partial F}{\partial\gamma}(g_0,\gamma_0)=0$ if and only if $\gamma_0$ is a $g_0$-geodesic\footnote{%
By geodesic, we will always mean an \emph{affinely parameterized} geodesic.} in $M$ joining $p$ and $q$.
Given one such pair $(g_0,\gamma_0)$, the second derivative $\frac{\partial^2F}{\partial\gamma^2}$ at $(g_0,\gamma_0)$ is:
\begin{equation}\label{eq:dersecrespgamma}
\frac{\partial^2F}{\partial\gamma^2}(g_0,\gamma_0)(V,W)=\int_0^1g_0\big(\mathbf D^{g_0}V,\mathbf D^{g_0}W\big)+g_0\big(R^{g_0}(\dot\gamma_0,V)\,\dot\gamma_0,W\big)\,\mathrm dt,
\end{equation}
where $\mathbf D^{g_0}$ denotes the covariant derivative along $\gamma_0$ induced by the Levi--Civita connection $\nabla^{g_0}$ of $g_0$, and
$R^{g_0}$ is the curvature tensor of $\nabla^{g_0}$.
This is the classical \emph{index form} of $\gamma_0$ relatively to the metric $g_0$.
\begin{lem}\label{thm:fredholmness}
$\frac{\partial^2F}{\partial\gamma^2}(g_0,\gamma_0)$ is a \emph{Fredholm} symmetric bilinear form on $T_{\gamma_0}\Omega_{p,q}(M)$,
i.e., it is represented by a self-adjoint Fredholm operator on $T_{\gamma_0}\Omega_{p,q}(M)$ relatively to the inner
product \eqref{eq:definnerprTgamma}.
\end{lem}
\begin{proof}
For all $t\in[0,1]$, let $A_t:T_{\gamma(t)}M\to T_{\gamma(t)}M$ be the $g_{\mathrm R}$-symmetric automorphism
such that $g_0=g_{\mathrm R}(A_t\cdot,\cdot)$ on $T_{\gamma(t)}M$. The map $\Phi:T_{\gamma_0}\Omega_{p,q}(M)\ni V\mapsto\widetilde V\in T_{\gamma_0}\Omega_{p,q}(M)$
defined by $\widetilde V(t)=A_tV(t)$ is an isomorphism; we will show that $\frac{\partial^2F}{\partial\gamma^2}(g_0,\gamma_0)$ is represented
relatively to the to the inner product \eqref{eq:definnerprTgamma} by an operator which is a compact perturbation of $\Phi$. Namely,
the difference $E(V,W)=\frac{\partial^2F}{\partial\gamma^2}(g_0,\gamma_0)(V,W)-\llangle\Phi V,W\rrangle$ is easily computed as:
\begin{multline*}E(V,W)=\int_0^1\Big[-g_{\mathrm R}\big(A'V,\mathbf D^{\mathrm R}W)+
g_{\mathrm R}\big(A\Gamma^{\mathrm R}V,\mathbf D^{\mathrm R}W\big)+g_{\mathrm R}\big(A\mathbf D^{\mathrm R}V,\Gamma^{\mathrm R}W\big)
\\ +g_{\mathrm R}(A\Gamma^{\mathrm R}V,\Gamma^{\mathrm R}W)+g_{\mathrm R}(ARV,W)\Big]\,\mathrm dt,
\end{multline*}
where $\Gamma^{\mathrm R}=\mathbf D^{g_0}-\mathbf D^{\mathrm R}$ is the Christoffel tensor of $\nabla^{g_0}$ relatively to $\nabla^{\mathrm R}$.
Each term in the right hand side of the above equality is bilinear in $(V,W)$, and does not contain any derivative of at least
one of its two arguments, i.e., it is continuous relatively to the $C^0$-topology in one of its arguments.
From the compactness of the inclusion $H^1\hookrightarrow C^0$, it follows easily that $E$ is represented by a compact operator
on $T_{\gamma_0}\Omega_{p,q}(M)$.
\end{proof}
The kernel of the index form $\frac{\partial^2F}{\partial\gamma^2}(g_0,\gamma_0)$ is the space of all  Jacobi fields $J$ along $\gamma_0$ such that $J(0)=J(1)=0$.
The second mixed derivative $\frac{\partial^2F}{\partial g\partial\gamma}$ is computed as follows; let $\left]-\varepsilon,\varepsilon\right[\ni s\mapsto g_s\in\mathcal A$
be a smooth variation of $g_0$, with $\frac{\mathrm d}{\mathrm ds}\big\vert_{s=0}g_s=h\in\mathcal E$.
As we have seen in Subsection~\ref{sub:geompreliminaries}, in order to perform this computation we will fix
an arbitrary symmetric connection $\nabla$ on $M$;  we will make a specific choice of such connection when needed
(see proof of Proposition~\ref{thm:gensemiRiem}).
Using the Christoffel tensor
$\Gamma^{g_s}$ of the metric $g_s$ relatively to $\nabla$ (see \eqref{eq:Christoffeltensor}), we compute:
\begin{multline}\label{eq:contotrasversalita}
\frac{\partial^2F}{\partial g\partial\gamma}(g_0,\gamma_0)(h,V)=
\frac{\mathrm d}{\mathrm ds}\Big\vert_{s=0}\int_0^1g_s\big(\dot\gamma_0,\mathbf D^{g_s}V\big)\,\mathrm dt\\=
\frac{\mathrm d}{\mathrm ds}\Big\vert_{s=0}\int_0^1g_s\big(\dot\gamma_0,\mathbf DV\big)+g_s\big(\dot\gamma,\Gamma^{g_s}(\dot\gamma,V)\big)\,\mathrm dt
\\ =\int_0^1h(\dot\gamma_0,\mathbf DV)\,\mathrm dt+\tfrac12\frac{\mathrm d}{\mathrm ds}\Big\vert_{s=0}\int_0^1\nabla g_s(V,\dot\gamma_0,\dot\gamma_0)
+\nabla g_s(\dot\gamma_0,\dot\gamma_0,V)-\nabla g_s(\dot\gamma_0,V,\dot\gamma_0)\,\mathrm dt\\
=\int_0^1h(\dot\gamma_0,\mathbf DV)+\tfrac12\nabla h(V,\dot\gamma_0,\dot\gamma_0)\,\mathrm dt.
\end{multline}
We will need to study the self intersections of geodesics, and the following elementary result
will be useful:
\begin{lem}\label{thm:selfintersections}
Let $(M,g)$ be a semi-Riemannian manifold, and let $\gamma:[0,1]\to M$ be a geodesic. Then, the set:
\[\big\{(s,t)\in[0,1]\times[0,1]:s\ne t,\ \gamma(s)=\gamma(t)\big\}\]
is finite, unless $\gamma$ is a closed geodesic with period $T<1$.
\end{lem}
\begin{proof}
Assume the existence of sequences $s_n$ and $t_n$ in $[0,1]$, with $s_n\ne t_n$, $\gamma(s_n)=\gamma(t_n)$
and $s_i\ne s_j$ for all $i\ne j$ (otherwise the pairs $(s_n,t_n)$ would be a finite number).
Because of the local injectivity of $\gamma$ we can assume that  $t_i\ne t_j$ for all $i\ne j$, and
up to taking subsequences, that $\lim s_n=s$ and $\lim t_n=t$, with $s,t\in[0,1]$; we can also assume that
$s_n\ne s$ and $t_n\ne t$ for all $n$. Clearly, $\gamma(s)=\gamma(t)$; since $\gamma$ is locally
injective (it is an immersion), then it must be $s\ne t$, say $t>s$.
Set $\mu(r)=\gamma(r-t+s)$; this is a geodesic, defined for $r$ in a neighborhood of $t$, and such that $\mu(t)=\gamma(t)$.
Moreover, set $t_n'=s_n-s+t$; this is a sequence converging to $t$, and with $t_n'\ne t$ for all $n$.
We have $\mu(t_n')=\gamma(t_n)$ for all $n$, and this implies that the tangent vectors $\dot\mu(t)=\dot\gamma(s)$ and $\dot\gamma(t)$ are
linearly dependent. Since $\gamma$ is affinely parameterized, it must be $\dot\gamma(s)=\dot\gamma(t)$, which
implies that $\gamma$ is periodic with period $T=t-s\le1$. It can't be $T=1$, i.e., $s=0$ and $t=1$, because
otherwise it would be $\gamma(t_n)=\gamma(s_n)=\gamma(s_n+1)$ for all $n$, with $t_n<1$ and $s_n+1>1$ converging to $1$,
contradicting the local injectivity of $\gamma$ around $1$.
\end{proof}
\begin{prop}\label{thm:gensemiRiem}
Let $M$ be a smooth manifold, let $\mathcal E\subset\mathbf\Gamma_{\mathrm{sym}}^k(TM^*\otimes TM^*)$ be a $C^k$-Whitney type Banach space of tensors over $M$, with $k\ge2$, let $\nu\in\{0,\ldots,\Dim(M)\}$ be fixed
and let $\mathcal A\subset\mathcal E\cap\mathrm{Met}^k_\nu(M)$ be an open subset of $\mathcal E$.
Given any pair of \emph{distinct} points $p,q\in M$, the set of semi-Riemannian metrics $g\in\mathcal A$
such that all $g$-geodesics joining $p$ and $q$ are nondegenerate, is generic in $\mathcal A$.
\end{prop}
\begin{proof}
We will prove the result as application of Corollary~\ref{thm:genericnondegeneracy} to the geodesic setup above.
In view of the Fredholmness result of Lemma~\ref{thm:fredholmness},
we only need to check that the transversality condition \eqref{eq:transvcond} is satisfied in this context.
We need to prove that,
given a semi-Riemannian metric $g_0\in\mathcal A$, a $g_0$-geodesic $\gamma_0$ joining $p$ and $q$, and
a non trivial $g_0$-Jacobi field $V$ along $\gamma_0$, with $V_0=V_1=0$, then there exists $h\in\mathcal E$
for which the quantity in the last term of \eqref{eq:contotrasversalita} does not vanish.
We will find such an $h$ to be a symmetric $(0,2)$-tensor of class $C^k$ having compact support in $M$, and thus
$h\in\mathcal E$.
Assume first that $\gamma_0$ is not a portion of a closed geodesic
in $M$ with minimal period $T<1$. Then, by Lemma~\ref{thm:selfintersections}, $\gamma_0$ has at most a finite number of self-intersections.
We can therefore find an open subinterval $I\subset[0,1]$ with the following properties:
\begin{itemize}
\item[(a)] $t\in I$ and $s\not\in I$ implies $\gamma_0(s)\ne\gamma_0(t)$;
\item[(b)] $V_t$ is not parallel to $\dot\gamma_0(t)$ for all $t\in I$.
\end{itemize}
As to property (b), observe that since $V$ is a nontrivial Jacobi field which vanishes at the endpoints,
then it is not everywhere multiple of $\dot\gamma_0$, and by Lemma~\ref{thm:Jnontparalleltogamma} the set
of instants $t$ at which $V_t$ is parallel to $\dot\gamma_0(t)$ is finite. Choose now an open subset $U\subset M$
containing $\gamma_0(I)$ and such that
\begin{equation}\label{eq:interUimmgamma}
\gamma_0(t)\in U\cap\gamma_0\big([0,1]\big)\quad\Longleftrightarrow\quad t\in I;
\end{equation}
for instance, take $U$ to be the complement of the compact set $\gamma_0\big([0,1]\setminus I\big)$.
We will now use the result of Lemma~\ref{thm:constrloc-strong} applied to the case of
symmetric $(0,2)$-tensor fields, as follows. For $t\in I$, we choose
$H_t$ identically zero, and $K_t$ a symmetric bilinear form on $T_{\gamma_0(t)}M$ (depending smoothly on $t$)
such that $K_t\big(\dot\gamma_0(t),\dot\gamma_0(t)\big)\ge0$
with $\int_IK_t\big(\dot\gamma_0(t),\dot\gamma_0(t)\big)\,\mathrm dt>0$. By possibly reducing the
size of the interval $I$, we can assume that the thesis of Lemma~\ref{thm:constrloc-strong} applies, and
we get a globally defined smooth symmetric $(2,0)$-tensor $h$ on $M$, having compact support contained
in $U$, such that $h_{\gamma_0(t)}=0$ and $\nabla_{V_t}h=K_t$ for all $t\in I$.
For such $h$, by \eqref{eq:interUimmgamma} we have:
\[\int_0^1\Big[h(\dot\gamma_0,\mathbf DV)+\tfrac12\nabla h(V,\dot\gamma_0,\dot\gamma_0)\Big]\,\mathrm dt=\tfrac12\int_IK_t\big(\dot\gamma_0(t),\dot\gamma_0(t)\big)\,\mathrm dt>0,\]
which concludes the proof when $\gamma_0$ is not periodic of period $T<1$.

Assume now that $\gamma_0$ is periodic, of period $T<1$. Consider the following numbers:
\[t_*=\min\big\{t>0:\gamma_0(t)=q\big\},\quad k_*=\max\big\{k\in\Z:kT<1\big\},\]
for which the following hold:\footnote{%
Here the assumption that $p\ne q$ is being used. Note that if $p=q$, then $t_*=T$, and the argument below
fails.}
\[k_*\ge1,\quad 0<t_*<T,\quad 1=k_*T+t_*.\]
The geodesics $\gamma_1=\gamma_0\vert_{[0,t_*]}$ and $\gamma_2=\gamma_0\vert_{[t_*,T]}$ join $p$ and $q$ ($\gamma_2$ with the opposite
orientation),
and the first part of the proof applies to both $\gamma_1$ and $\gamma_2$. Thus,
we can find open intervals $I_1=[a_1,b_1]\subset[0,t_*]$ and $I_2=[a_2,b_2]\subset[t_*,T]$ such that:
\begin{itemize}
\item[(a1)] $t\in I_1$, $s\in\big([0,t_*]\setminus I_1\big)\cup [t_*,T]$ implies $\gamma_0(s)\ne\gamma_0(t)$;
\item[(a2)] $t\in I_2$, $s\in\big([t_*,T]\setminus I_2\big)\cup [0,t_*]$ implies $\gamma_0(s)\ne\gamma_0(t)$.
\end{itemize}
We can also find open subsets $U_1,U_2\subset M$, with $\gamma(I_i)\subset U_i$, $i=1,2$,
satisfying:
\begin{equation}\label{eq:proprietaapertiUi}
\begin{aligned}
&\gamma_0(t)\in U_1\cap\gamma_0(I_1)\quad\Longleftrightarrow\quad \exists\,r\in\{0,\ldots,k_*\}\ \text{such that}\ t-rT\in I_1,
\\
&\gamma_0(t)\in U_2\cap\gamma_0(I_2)\quad\Longleftrightarrow\quad \exists\,r\in\{0,\ldots,k_*-1\}\ \text{such that}\ t-rT\in I_2.
\end{aligned}
\end{equation}
For $j=1,2$, consider the orthogonal Jacobi field $W^j$ along $\gamma_j$ defined by:
\begin{equation}\label{eq:defW1W2}
W^1_t=\sum_{r=0}^{k_*}V_{t+rT},\quad W^2_t=\sum_{r=0}^{k_*-1}V_{t+rT}.
\end{equation}
It is not the case that both $W^1$ and $W^2$ are everywhere parallel to $\dot\gamma_0$ on $I_1$ and $I_2$ respectively,
for otherwise from \eqref{eq:defW1W2} one would conclude easily that $V$ would be everywhere parallel to $\dot\gamma_0$
(Lemma~\ref{thm:Jnontparalleltogamma}).
Assume that, say, $W^1$ is not everywhere parallel to $\dot\gamma_0$ on $I_1$, i.e.,
by Lemma~\ref{thm:Jnontparalleltogamma}, there are only isolated values of $t$ where $W^1_t$ is parallel to $\dot\gamma_0(t)$;
the other case is totally analogous. By reducing the size of $I_1$, we can  assume
that $W^1_t$ is never a multiple of $\dot\gamma_0(t)$ on $I_1$. Now, the first part of the
proof can be repeated, by replacing the Jacobi field $V$ with $W^1$. We can find a globally defined
symmetric $(0,2)$-tensor $h$ on $M$, with compact support contained in  $U_1$,
with prescribed value $H$ and covariant derivative $K$ in the direction $W^1$ along $\gamma_0\vert_{I_1}$.
Choose $H$ and $K$ as above, and compute:
\begin{multline*}\int_0^1h(\dot\gamma_0,\mathbf DV)+\tfrac12\nabla h(V,\dot\gamma_0,\dot\gamma_0)\,\mathrm dt=
\tfrac12\sum_{r=0}^{k_*}\int_{a_1+rT}^{b_1+rT}\nabla h(V,\dot\gamma_0,\dot\gamma_0)\,\mathrm dt\\=
\tfrac12\int_{a_1}^{b_1}\nabla h(W^1,\dot\gamma_0,\dot\gamma_0)\,\mathrm dt=\tfrac12\int_IK_t\big(\dot\gamma_0(t),\dot\gamma_0(t)\big)\,\mathrm dt>0.
\end{multline*}
This concludes the proof.
\end{proof}
\subsection{Perturbations of a metric in its conformal class}\label{sub:conformal}
It is a natural question to ask whether the genericity result of Proposition~\ref{thm:gensemiRiem}
remains true if one consider more restrictive classes of variations of a given metric. Particularly interesting
examples are perturbations inside a given conformal class of semi-Riemannian metrics.
However, one cannot expect that the genericity result holds in this case, as the following example shows.
\begin{example}\label{exa:lighlikeperturbations}
Let $(M,g_0)$ be a semi-Riemannian manifold, and let $\gamma:[0,1]\to M$ be a lightlike geodesic
in $M$ with $p=\gamma(0)$ and $q=\gamma(1)$ conjugate along $\gamma$. Then, given any semi-Riemannian
metric $g$ on $M$ which is conformal to $g_0$, there exists a suitable reparameterization
$\widetilde\gamma$ of $\gamma$ which is a lightlike $g$-geodesic, and
such that $p$ and $q$ are conjugate\footnote{%
In the Lorentzian case, conjugate points along lightlike geodesics are preserved even
by maps more general than conformal diffeomorphisms.
It is not hard to prove (for instance, via bifurcation theory using \cite[Corollary~11]{JavPic}) the following:
\begin{lem-n}
Let $(M_i,g_i)$, $i=1,2$, be Lorentzian manifolds, and let $\Psi:M_1\to M_2$ be a continuous injective map
that carries timelike curves to timelike curves and lightlike pre-geodesic to lightlike pre-geodesics.
Then, $\Psi$ carries pairs of conjugate points along lightlike geodesics into pairs of
conjugate points along lightlike geodesics.
\end{lem-n}
Note that if $\Psi$ as in the statement of the Lemma is a diffeomorphism,
then necessarily $\Psi$ is conformal, by a well known result of Dajczer and Nomizu, see \cite{DaNo}.}
along $\widetilde\gamma$ (see for instance \cite[Theorem~2.36]{MinSan}). Thus, conformal perturbations do not destroy degeneracy of lightlike geodesics.
\end{example}

We will show that, \emph{apart from the lightlike case}, generic conformal perturbations
are sufficient to destroy degeneracy. In view of Example~\ref{exa:lighlikeperturbations},
this is the best possible result.

Given a semi-Riemannian
metric tensor $\bar g$ on $M$ of class $C^k$, $k\ge2$, let us denote by $\mathfrak C^k(\bar g)$ the set of all
semi-Riemannian metrics on $M$ that are globally conformal to $\bar g$, i.e., the set of metrics of the form
$g=\psi\cdot\bar g$ for some function $\psi:M\to\R^+$ of class $C^k$. As above, when $M$ is not compact,
there is no natural topological structure on $\mathfrak C^k(\bar g)$ that makes it homeomorphic to an open subset of
a Banach space. Let us denote by $C^k(M)$ the vector space of all real valued $C^k$-functions on $M$.
In analogy with the notion of $C^k$-Whitney type Banach spaces of tensor fields, let us
call a \emph{$C^k$-Whitney type Banach space of functions} on $M$ a vector subspace $\mathcal F$ of
$C^k(M)$ endowed with a Banach space norm $\Vert\cdot\Vert_{\mathcal F}$ satisfying:
\begin{itemize}
\item[(a)] $\mathcal F$ contains all the functions in $C^k(M)$ having compact support;
\item[(b)] $\Vert\cdot\Vert_{\mathcal F}$-convergence implies $C^k$-convergence on compact subsets of $M$.
\end{itemize}
For instance, given a complete Riemannian metric $g_{\mathrm R}$ on
$M$, a $C^k$-Whitney type Banach space of functions on $M$ can be obtained by setting $\mathcal F=\mathfrak C^k(M;g_{\mathrm R})$,
which consists of all functions in $C^k(M)$ that have $g_{\mathrm R}$-bounded
derivatives up to order $k$.

Given a $C^k$-Whitney type Banach space $\mathcal F$ of functions on $M$ and a semi-Riemannian
metric tensor $\bar g$ on $M$, let us denote by $\mathfrak C^k(\bar g;\mathcal F)$
the set:
\[\mathfrak C^k(\bar g;\mathcal F)=\big\{\psi\cdot\bar g:\psi\in\mathcal F\big\}.\]
and by $\mathfrak C^k_+(\bar g;\mathcal F)$ the \emph{$\mathcal F$-conformal class} of $\bar g$, defined by:
\[\mathfrak C^k_+(\bar g;\mathcal F)=\big\{\psi\cdot\bar g:\psi\in\mathcal F,\ \psi>0\big\}.\]
The map $\psi\mapsto\psi\cdot\bar g$ gives an identification of
the set $\mathfrak C^k(\bar g,\mathcal F)$ with the Banach space $\mathcal F$
(and of $\mathfrak C^k_+(\bar g,\mathcal F)$ with the subset $\mathcal F_+$
of everywhere positive functions of $\mathcal F$); $\mathfrak C^k(\bar g,\mathcal F)$ will be thought as
a metric space with the induced norm.
\begin{prop}
Let $M$ be a smooth manifold, $\bar g$ a semi-Riemannian metric tensor on $M$ of class $C^k$, $k\ge2$,
and let $p,q\in M$ be fixed distinct points. Let $\mathcal F\subset C^k(M)$ be a $C^k$-Whitney type Banach space
of functions on $M$, and let $\mathcal A$ be a (non empty) open subset of $\mathfrak C^k(\bar g;\mathcal F)$
contained in $\mathfrak C^k_+(\bar g;\mathcal F)$.
Then, the set of metrics $g\in\mathcal A$ such that every nonlightlike $g$-geodesic in $M$
joining $p$ and $q$ is nondegenerate is generic in $\mathcal A$.
\end{prop}
\begin{proof}
Let $g_0\in\mathcal A$ and $\gamma_0$ be a non  lightlike, i.e., $g_0(\dot\gamma_0,\dot\gamma_0)\ne0$,
$g_0$-geodesic in $M$ joining $p$ and $q$; let $V$ be a nontrivial $g_0$-Jacobi field along $\gamma_0$ that vanishes at both endpoints.
We will find a variation $h$  of the form $\psi\cdot g_0$, with $\psi:M\to\R$ a smooth nonnegative function
with \emph{small} compact support, and for which the last term in \eqref{eq:contotrasversalita} does not vanish.
For such a variation $h$, the last term of \eqref{eq:contotrasversalita} is easily
computed by choosing $\nabla$ to be the Levi--Civita connection of $g_0$. Namely, in this case $g_0(\dot\gamma_0,\mathbf DV)$ vanishes identically;
this is because the function $g_0(\dot\gamma_0,V)$ is affine, and since it vanishes at $0$ and at $1$, it must be identically
zero, as well as its derivative $g_0(\dot\gamma_0,\mathbf DV)$. Thus, for such a variation $h$,
the quantity $h(\dot\gamma_0,\mathbf DV)$ vanishes identically. Moreover, since $\nabla g_0=0$, then
$\nabla h(V,\dot\gamma_0,\dot\gamma_0)=V(\psi)\cdot g_0(\dot\gamma_0,\dot\gamma_0)$. Since we are assuming that
the constant $g_0(\dot\gamma_0,\dot\gamma_0)$ is not null, we have now reduced the problem to determining a smooth
nonnegative function $\psi$ with the property that $\int_0^1V\big(\psi(\gamma_0(t))\big)\,\mathrm dt\ne0$; we want such
a function $\psi$ with compact support in $M$.
For the construction of such $\psi$, the procedure is analogous to that
in the proof of Proposition~\ref{thm:gensemiRiem}, using Lemma~\ref{thm:constrloc-strong}.
Assume first that $\gamma_0$ is not a portion of a closed geodesic
in $M$ with minimal period $T<1$. Then, by Lemma~\ref{thm:selfintersections}, $\gamma_0$ has at most a finite number of self-intersections.
and we can find an open subinterval $I\subset[0,1]$ satisfying properties (a) and (b) in the proof of Proposition~\ref{thm:gensemiRiem},
and an open subset $U\subset M$ containing $\gamma_0(I)$ and such that \eqref{eq:interUimmgamma} holds.
Now, choose a smooth function $\alpha:I\to\R$ having compact support and such that $\int_I\alpha(t)\,\mathrm dt>0$.
By Lemma~\ref{thm:constrloc-strong} (applied to the case of the trivial vector bundle $\mathcal E$ over $M$ whose fiber is one dimensional),
we can find a smooth map $\psi:M\to\R$ having compact support contained in $U$,  such that
$\psi\big(\gamma_0(t)\big)=1$ and $V_t\big(\psi)=\alpha(t)$ for all $t\in I$.
With this choice, we have:
\begin{equation}\label{eq:primauguaglianza}
\int_0^1V_t(\psi)\,\mathrm dt\stackrel{\text{by\ \eqref{eq:interUimmgamma}}}=\int_IV_t(\psi)\,\mathrm dt=\int_I\alpha(t)\,\mathrm dt>0.
\end{equation}
This concludes the proof in the case that $\gamma_0$ is not a portion of a closed geodesic.
When $\gamma_0$ is periodic with period $T<1$, the construction is totally analogous to the proof
of Proposition~\ref{thm:gensemiRiem}. One defines Jacobi fields $W^1$ and $W^2$ as in \eqref{eq:defW1W2},
open intervals $I_i\subset[0,1]$ and open subsets $U_i\subset M$ satisfying \eqref{eq:proprietaapertiUi};
by the same arguments, one obtains that at least one of two Jacobi fields, say $W^i$, is never parallel to $\dot\gamma_0(t)$
on $I_i$.
Define $\psi:M\to\R$ as above replacing the Jacobi field $V$ with $W^i$ and the interval $I$ with $I_i$
using a smooth function $\alpha:I_i\to\R$ with compact support and satisfying $\int_{I_i}\alpha(t)\,\mathrm dt>0$.
As above, set $h=\psi\cdot g_0$; now, \eqref{eq:primauguaglianza} is replaced by:
\[
\int_0^1V_t(\psi)\,\mathrm dt=\sum_{r=0}^{k_*-i+1}\int_{I_i}V_{t+rT}(\psi)\,\mathrm dt
=\int_{I_i}W^i_t(\psi)\,\mathrm dt=\int_{I_i}\alpha(t)\,\mathrm dt>0.\qedhere
\]
\end{proof}
\subsection{Orthogonally split metrics}
\label{sub:splitting}
Let us now take a product manifold $M=M_1\times M_2$, with $\Dim(M_i)=n_i$, $i=1,2$, and consider the subset
$\mathrm{Met}_{\mathrm{split}}^k(M_1,M_2)$ of $\mathrm{Met}^k_{n_2}(M)$
consisting of all symmetric $(0,2)$-tensors $g$ of class $C^k$ on $M$ such that:
\begin{itemize}
\item[(a)] $g_{(x,y)}\big((v_1,0),(0,v_2)\big)=0$;
\item[(b)] $g_{(x,y)}$ is positive definite on $T_xM_1\times\{0\}$;
\item[(c)] $g_{(x,y)}$ is negative definite\footnote{In fact, rather than (b) and (c),
we will use the weaker assumptions that $g$ is nondegenerate on $TM_1\times\{0\}$ and on $\{0\}\times TM_2$.} on $\{0\}\times T_yM_2$,
\end{itemize}
for all $(x,y)\in M_1\times M_2$, all $v_1\in T_{x}M_1$ and all $v_2\in T_yM_2$. Elements of $\mathrm{Met}_{\mathrm{split}}^k\!(\!M_1,\!M_2\!)$
will be called \emph{orthogonally split} semi-Riemannian metric tensors on $M_1\times M_2$.
More generally, a $(0,2)$-tensor field $\mathfrak b$ on $M$ will be called orthogonally split
if it satisfies \[\mathfrak b_{(x,y)}\big((v_1,0),(0,v_2)\big)=0\] for all $(x,y)\in M_1\times M_2$, all $v_1\in T_{x}M_1$ and all $v_2\in T_yM_2$.

Let $\mathcal E\subset\mathbf\Gamma^k_{\mathrm{sym}}(TM^*\otimes TM^*)$ be a $C^k$-Whitney type Banach space
of tensors on $M$; we will denote by $\mathrm{Met}_{\mathrm{split}}^k(M_1,M_2;\mathcal E)$
the intersection $\mathrm{Met}_{\mathrm{split}}^k(M_1,M_2)\cap\mathcal E$. Note that the set $\mathcal E_{\mathrm{split}}$
consisting of all orthogonally split tensor fields in $\mathcal E$ is a (non trivial) closed subspace of $\mathcal E$.
Non triviality follows from the fact that $\mathcal E_{\mathrm{split}}$ contains all
the orthogonally split tensor fields on $M$ having compact support.

\begin{prop}\label{thm:gensemiRiem-split}
Let $M_1$ and $M_2$ be smooth manifolds, let $\mathcal E$ be a $C^k$-Whitney type Banach space
of tensors on the product $M=M_1\times M_2$, and let $\mathcal A$ be an open subset of $\mathcal E_{\mathrm{split}}$
with $\mathcal A\subset\mathrm{Met}_{\mathrm{split}}^k(M_1,M_2;\mathcal E)$. Given any two distinct points $p,q\in M$, then the set of all
$g\in\mathcal A$  such that all $g$-geodesics
in $M$ joining $p$ and $q$ are nondegenerate is generic in $\mathcal A$.
\end{prop}
\begin{proof}
Let $g_0\in\mathcal A$ be fixed and consider a $g_0$-geodesic
$\gamma_0=(x_1,x_2)$ joining $p$ and $q$, and a nontrivial  $g_0$-Jacobi field $V=(V_1,V_2)$
along $\gamma_0$ which vanishes at the endpoints. The proof goes along the same lines as the proof of
Proposition~\ref{thm:gensemiRiem}, with the difference that here the variation $h$ has to be found
in the Banach space $\mathcal E_{\mathrm{split}}$. Again, we will determine the variation $h$ to be
an orthogonally split symmetric $(0,2)$-tensor field having compact support in $M$.
One has to repeat the proof of Proposition~\ref{thm:gensemiRiem}, which involves the
construction of a family of bilinear forms $K_t$ on $T_{\gamma_0(t)}M=T_{x_1(t)}M_1\oplus T_{x_2(t)}M_2$
with the property that $\int_IK_t\big(\dot\gamma_0(t),\dot\gamma_0(t)\big)\,\mathrm dt>0$ on some given
interval $I$.
Recall that in the proof of Proposition~\ref{thm:gensemiRiem} we are choosing the family $H_t$ to vanish identically.
In the case under consideration, the desired $K_t$ can be chosen such that $K_t\big((v_1,0),(0,v_2)\big)=0$
for every $v_1\in T_{x_1(t)}M_1$, $v_2\in T_{x_2(t)}M_2$ and every $t\in I$. Namely, it suffices
to choose families of symmetric bilinear forms $K_t^i$ on $T_{x_i(t)}M_i$, $i=1,2$,
satisfying
\begin{equation}\label{eq:sommapositiva}
\sum_{i=1}^2\int_IK^i_t\big(\dot x_i(t),\dot x_i(t)\big)\,\mathrm dt>0
\end{equation}
and set $K_t\big((v_1,v_2),(w_1,w_2)\big)=K^1_t(v_1,w_1)+K^2_t(v_2,w_2)$ for all $t$.
The existence of families $K_t^i$ that satisfy \eqref{eq:sommapositiva} is easily proven,
keeping in mind that $\dot x_1(t)$ and $\dot x_2(t)$ are not both zero anywhere.
Now, Lemma~\ref{thm:constrloc-strong} is applied to the vector bundle $E$ over $M$ whose sections
are the symmetric $(0,2)$-tensors $h$ on $M$ satisfying $h_{(x_1,x_2)}\big((v_1,0),(0,v_2)\big)=0$ for
all $x_i\in M_i$ and all $v_i\in T_{x_i}M_i$. In order to make the result of Lemma~\ref{thm:constrloc-strong}
compatible with formula~\eqref{eq:contotrasversalita}, one more detail needs to be clarified.
Namely, one needs to consider a connection $\nabla$ in $E$ which is \emph{inherited} from a
connection $\widetilde\nabla$ in $TM$; more precisely, $\nabla$ has to be given as the restriction
to the subbundle $E$ of the induced connection $\widetilde\nabla$ on $TM^*\otimes TM^*$.
It will not be the case in general that connections on $TM^*\otimes TM^*$
restrict to $E$, i.e., that covariant derivatives of sections of $E$ remain in $E$.
In order to make the connection $\widetilde\nabla$ restrictable to $E$, the corresponding
connection $\widetilde\nabla$ on $TM$ has to be chosen of the form:
\[\widetilde\nabla=\pi_1^*(\nabla^1)\oplus\pi_2^*(\nabla^2),\]
where $\nabla^i$ is a connection on $TM_i$, and $\pi_i:M_1\times M_2\to M_i$ is the projection,
$i=1,2$. This concludes the argument.
\end{proof}
\subsection{Globally hyperbolic Lorentzian metrics}\label{sub:globhyper}
Let us now study the nondegeneracy problem for geodesics in \emph{globally hyperbolic} Lorentzian manifolds.
A time oriented Lorentz\-ian metric $g$ on a connected manifold $M$ is said to be globally hyperbolic
if $(M,g)$ admits  a \emph{Cauchy surface} $\Sigma$, i.e., $\Sigma$ is a spacelike hypersurface of $M$ which is
met exactly once by every non extendible causal curve. There are several equivalent notions of global
hyperbolicity that will not be discussed here (see \cite{BeErEa96, BeSan2, One83} for details). Let us
recall that by a classical result by Geroch \cite{Ger}, whose statement has been recently strengthened
by Bernal and S\'anchez in \cite{BerSan2, BerSan3}, a globally hyperbolic Lorentzian manifold $(M,g)$
is isometric to a product $\Sigma\times\R$, where $\Sigma$ is any Cauchy surface of $(M,g)$, endowed with
an orthogonally split metric tensor which is positive definite on the factor $\Sigma$ and negative definite on the
one-dimensional factor $\R$. We will then consider a manifold $M$ of the form $\Sigma\times\R$, where $\Sigma$ is
a smooth manifold endowed with a complete Riemannian metric $g^0$; we will denote by $\pi_\Sigma:\Sigma\times\R\to\Sigma$
the projection
onto the first factor. We will study the set of
metrics $g^{\alpha,\beta}$ on $M$, where:
\begin{itemize}
\item $\alpha$ is a fixed smooth section of the pull-back bundle $\pi_\Sigma^*\big(T\Sigma^*\otimes T\Sigma\big)$
such that $g^0_x\big(\alpha_{(x,s)}\cdot,\cdot\big)$ is positive
definite on $T_x\Sigma$ for all $x\in\Sigma$ and all $s\in\R$;
\item $\beta:\Sigma\times\R\to\R^+$ is a smooth positive function,
\end{itemize}
and the metric tensor $g^{\alpha,\beta}$ is defined by:
\begin{equation}\label{eq:defgalphabeta}
g^{\alpha,\beta}_{(x,s)}\big((v,r),(w,\bar r)\big)=g^0_x\big(\alpha_{(x,s)}v,w\big)-\beta_{(x,s)}r\bar r,
\end{equation}
for all $x\in\Sigma$, $s\in\R$, $v,w\in T_x\Sigma$, $r,\bar r\in T_t\R\cong\R$.
A genericity result totally analogous to Proposition~\ref{thm:gensemiRiem-split} holds for the
family of metrics $g^{\alpha,\beta}$, that can be described simply as metric of splitting type
on a product manifold $M_1\times M_2$ with $M_2$ one-dimensional.
We will be interested in studying the genericity of nondegeneracy property in the
subfamily of the $g^{\alpha,\beta}$ consisting of globally hyperbolic metrics.

Given $\alpha$ as above, set:
\[\lambda_{(x,s)}(\alpha)=\Vert\alpha_{(x,s)}^{-1}\Vert^{-1},\]
where $\Vert\cdot\Vert$ denotes the operator norm on $\mathrm{End}(T_x\Sigma)$ induced by the
positive definite inner product
$g^0_x$. Equivalently, $\lambda_{(x,s)}(\alpha)$ can be defined as the minimum eigenvalue of the
positive operator $\alpha_{(x,s)}$ on $T_x\Sigma$. Sufficient conditions for the global hyperbolicity
of the Lorentzian metric $g^{\alpha,\beta}$ have been studied in the literature, see \cite{SanBari};
we will be interested in the following:
\begin{prop}\label{thm:condglobhyperb}
Let $x_0$ be any fixed point in $\Sigma$, and denote by $\mathrm d_0:\Sigma\to\left[0,+\infty\right[$ be
the distance from $x_0$ function induced by the Riemannian metric $g^0$.
Assume that for all integer $n>0$ the following holds:
\[\sup_{\stackrel{x\in\Sigma}{\vert s\vert\le n}}\sqrt{\frac{\beta_{(x,s)}}{\lambda_{(x,s)}(\alpha)\big(1+\mathrm d_0(x)^2\big)}}<+\infty.\]
Then, for all $s_0\in\R$, $\Sigma\times\{s_0\}$ is a Cauchy surface of $g^{\alpha,\beta}$.
In particular, if $\Sigma$ is compact then $g^{\alpha,\beta}$ is always globally hyperbolic.
\end{prop}
\begin{proof}
See \cite[Proposition~3.2]{SanBari}.
\end{proof}
Motivated by the result above, let us consider the  Banach space $\mathcal G$ whose points are pairs $(\alpha,\beta)$,
where:
\begin{itemize}
\item $\alpha$ is a section of class $C^2$ of the vector bundle $\pi_\Sigma^*(T\Sigma^*\otimes T\Sigma)$
such that $\alpha_{(x,s)}$
is a $g_0$-symmetric operator on $T_x\Sigma$ for all $(x,s)$;
\item $\beta:\Sigma\times\R\to\R$ is a map of class $C^2$;
\item $\alpha$ satisfies the following boundedness assumptions:\smallskip

\begin{itemize}
\item[$\diamond$] $C_0(\alpha)=\sup\limits_{(x,s)\in\Sigma\times\R}\Vert\alpha_{(x,s)}
\big(1+\mathrm d_0(x^2)\big)\Vert<+\infty$. Here, $\Vert\cdot\Vert$ is
the operator norm on $T_x\Sigma$ induced by the Riemannian metric $g_0$.\smallskip

\item[$\diamond$] $C_1(\alpha)=\sup\limits_{(x,s)\in\Sigma\times\R}\Vert\nabla\alpha_{(x,s)}\Vert<+\infty$.
Here, $\nabla$ is the connection
on the vector bundle $T^*(\Sigma\times\R)\otimes\pi_\Sigma^*(T\Sigma^*\otimes T\Sigma)$ induced by the Levi--Civita connection of $g_0$
and the standard connection on the factor $\R$.\smallskip

\item[$\diamond$] $C_2(\alpha)=\sup\limits_{(x,s)\in\Sigma\times\R}\Vert\nabla^2\alpha_{(x,s)}\Vert<+\infty$.
Here, the second covariant
derivative of $\alpha$ is taken relatively to the connection on the vector bundle $T^*(\Sigma\times\R)
\otimes T^*(\Sigma\times\R)\otimes\pi_\Sigma^*(T\Sigma^*\otimes T\Sigma)$ induced by the Levi--Civita connection of
$g_0$ and the standard connection on the factor $\R$.\smallskip
\end{itemize}
\item $\beta$ satisfies the following boundedness assumptions:
\smallskip

\begin{itemize}
\item[$\diamond$] $D_0(\beta)=\sup\limits_{(x,s)\in\Sigma\times\R}\vert\beta_{(x,s)}\vert<+\infty$.
\smallskip

\item[$\diamond$] $D_1(\beta)=\sup\limits_{(x,s)\in\Sigma\times\R}\Vert\mathrm d\beta_{(x,s)}\Vert<+\infty$.\smallskip

\item[$\diamond$] $D_2(\beta)=\sup\limits_{(x,s)\in\Sigma\times\R}\Vert\nabla\mathrm d\beta_{(x,s)}\Vert<+\infty$.
Here, $\nabla$ denotes the covariant derivative of the connection in $T^*(\Sigma\times\R)$ induced by the
Levi--Civita connection of $g_0$ and the standard connection on the factor $\R$.\smallskip
\end{itemize}
\end{itemize}
A Banach space norm on $\mathcal G$ is given by:
\[\Vert(\alpha,\beta)\Vert=\max\big\{C_0(\alpha),\ C_1(\alpha),\ C_2(\alpha),\ D_0(\beta),\ D_1(\beta),\ D_2(\beta)\big\}.\]
\begin{prop}
Let $\varepsilon$ and $b$ be fixed positive real numbers. The subset $\mathcal A_{\varepsilon,b}\subset\mathcal G$
given by:
\begin{multline*}
\mathcal A_{\varepsilon,b}=\Big\{(\alpha,\beta)\in\mathcal G:\text{$g_0(\alpha_{(x,s)}\cdot,\cdot)$ is positive definite},
\ \inf_{(x,s)\in\Sigma\times\R}\beta_{(x,s)}>0\\ \sup_{(x,s)\in\Sigma\times\R}\beta_{(x,s)}<b,
\
\text{and}\ \
\inf_{(x,s)\in\Sigma\times\R}\lambda_{(x,s)}(\alpha)\big((1+\mathrm d_0(x)^2\big)>\varepsilon\Big\}\end{multline*}
is open in $\mathcal G$. For all $(\alpha,\beta)\in\mathcal A_{\varepsilon,b}$, the tensor
$g^{\alpha,\beta}$ defined in \eqref{eq:defgalphabeta}
is a globally hyperbolic Lorentzian metric on $\Sigma\times\R$.
\end{prop}
\begin{proof}
As to the openness of $\mathcal A_{\varepsilon,b}$, the only non trivial question is establishing that the assumption
\begin{itemize}
\item $g_0(\alpha_{(x,s)}\cdot,\cdot)$ is positive definite
\item $\inf\limits_{(x,s)\in\Sigma\times\R}\lambda_{(x,s)}(\alpha)\big((1+\mathrm d_0(x)^2\big)>\varepsilon$
\end{itemize}
is open in the topology of $\mathcal G$. This follows immediately from the choice of the semi-norm $C_0(\alpha)$
above, and the fact that the ``least eigenvalue function'' $T\mapsto\lambda_{\mathrm{min}}(T)\in\R^+$
is Lipschitz with Lipschitzian constant $1$  in the set of positive symmetric operators $T$ on a vector space with inner product, that is, $|\lambda_{\mathrm{min}}(T)-\lambda_{\mathrm{min}}(\tilde{T})|\leq \|T-\tilde{T}\|$ (see also footnote $(4)$).

For $(\alpha,\beta)\in\mathcal A_{\varepsilon,b}$, the following inequality holds:
\begin{equation}\label{eq:garantisceglobhyerb}\sup_{\stackrel{x\in\Sigma}{s\in\R}}\sqrt{\frac{\beta_{(x,s)}}{\lambda_{(x,s)}(\alpha)\big(1+\mathrm d_0(x)^2\big)}}<
\sqrt{\frac b\varepsilon}<+\infty,
\end{equation}
and the global hyperbolicity of $g^{\alpha,\beta}$ is deduced from Proposition~\ref{thm:condglobhyperb}.
\end{proof}
\begin{prop}\label{thm:genglobhyperb}
Let $p$ and $q$ be distinct points in $\Sigma\times\R$.
For all $\varepsilon,b>0$, the set of pairs $(\alpha,\beta)\in\mathcal A_{\varepsilon,b}$ such that
$p$ and $q$ are not conjugate along any $g^{\alpha,\beta}$-geodesic in $\Sigma\times\R$ is generic in
$\mathcal A_{\varepsilon,b}$.
The open set:
\begin{multline*}
\mathcal A=\Big\{(\alpha,\beta)\in\mathcal G:\text{$g_0(\alpha_{(x,s)}\cdot,\cdot)$ is positive definite},
\ \inf_{(x,s)\in\Sigma\times\R}\beta_{(x,s)}>0\\ \sup_{(x,s)\in\Sigma\times\R}\beta_{(x,s)}<+\infty,
\
\text{and}\ \
\inf_{(x,s)\in\Sigma\times\R}\lambda_{(x,s)}(\alpha)\big((1+\mathrm d_0(x)^2\big)>0\Big\}\end{multline*}
contains a dense $G_\delta$ consisting of pairs $(\alpha,\beta)$ such that $p$ and $q$ are nonconjugate along any
$g^{\alpha,\beta}$-geodesic.
\end{prop}
\begin{proof}
The first statement follows from Proposition~\ref{thm:gensemiRiem-split}, observing that the vector space
$\mathcal E=\big\{g^{\alpha,\beta}:(\alpha,\beta)\in\mathcal G\big\}$ inherits from $\mathcal G$ a
Banach space norm that makes it into a $C^2$-Whitney type Banach space of orthogonally split
tensors over $\Sigma\times\R$.
Note that $\mathcal G$ contains all pairs $(\alpha,\beta)$ of class $C^2$ having compact support, and
its topology is finer than the weak Whitney $C^2$-topology. As to the second statement, it is enough to observe that
$\mathcal A$ can be described as the countable union $\bigcup_{n\ge1}\mathcal A_{\frac1n,n}$ of open
sets each of which contains a dense $G_\delta$ with the desired property.
\end{proof}
\subsection{Stationary Lorentzian metrics}
\label{sec:stationary}
Let us now consider the case of Lorentzian metrics admitting a timelike Killing vector field;
we will exhibit an example showing that the transversality condition discussed in Subsection~\ref{sub:semiRiemgeneric}
does not hold in general in this class.

Let $(M,g)$ be a Lorentzian manifold, and assume the existence of a Killing vector field $Y$ on $M$.
It is a simple observation that an integral line $\gamma$ of $Y$ is a geodesic in $(M,g)$ if and only if
at some point $\gamma(t_0)$ of $\gamma$ the function $g(Y,Y)$ has a critical point. Namely, since
$g(Y,Y)$ is invariant by the flow of $Y$, the existence of one critical point of $g(Y,Y)$ along $\gamma$
is equivalent to the fact that every point of $\gamma$ is critical for $g(Y,Y)$. Now,  $\gamma$ is a geodesic
if and only if $\nabla_YY=0$ along $\gamma$, i.e., if $g\big(\nabla_{Y(\gamma(t))}Y,v\big)=-g\big(\nabla_vY,Y\big)=-\frac12v\big(g(Y,Y)\big)=0$
for all $t$ and all $v\in T_{\gamma(t)}M$, i.e., if and only if $\gamma(t)$ is a critical point of
$g(Y,Y)$ for all $t$.
The geodesics in $(M,g)$ that are integral lines of $Y$ will be called \emph{vertical}.

Let us show that, given a Lorentzian manifold $(M,g)$ admitting a timelike Killing vector field $Y$, the transversality
condition may fail to hold along vertical geodesics in the class of all Lorentzian metrics
on $M$ that have the prescribed field $Y$ as timelike Killing vector field. A stationary Lorentzian manifold $(M,g)$ is said to be \emph{standard}
if $M$ is given by a product $M_0\times\R$, where $M_0$ is a differentiable manifold, and the metric tensor
$g$ is of the form:
\begin{equation}\label{eq:formastationary}g_{(x,s)}\big((v,r),(\bar v,\bar r)\big)=\mathfrak g_x(v,\bar v)+\mathfrak g_x\big(\delta(x),v\big)\bar r+
\mathfrak g_x\big(\delta(x),\bar v\big)r-\beta(x)r\bar r,\end{equation}
where $x\in M_0$, $s\in\R$, $v,\bar v\in T_xM_0$, $r,\bar r\in T_s\R\cong\R$, $\mathfrak g$ is a Riemannian metric tensor
on $M_0$, $\delta\in\mathfrak X(M_0)$ is a smooth vector field on $M_0$, and $\beta:M_0\to\R^+$ is a smooth positive function
on $M_0$. The field $Y=\partial_s$ tangent to the lines $\{x_0\}\times\R$, $x_0\in M_0$, is a timelike Killing vector
field in $(M,g)$; an immediate computation shows that
$g_{(x,s)}(Y,Y)=-\beta(x)$ for all $(x,s)\in M_0\times\R$. Locally, every stationary Lorentzian metric tensor has the form \eqref{eq:formastationary}.
When the vector field $\delta$ in \eqref{eq:formastationary} vanishes identically on $M_0$, then
the metric $\mathfrak g$ is said to be \emph{standard static}.

Let $\nabla$ be the Levi--Civita connection of the metric $\mathfrak g$ in $TM_0$; given a smooth map
$f_0:M_0\to\R$, denote by $\nabla f_0$ its gradient relatively to the metric $\mathfrak g$ and by
$\mathrm H^{f_0}(x):T_xM_0\to T_xM_0$, $x\in M_0$, the \emph{Hessian} of $f_0$ relatively to $\mathfrak g$ at the point
$x$, which is
the $\mathfrak g_x$-symmetric linear operator on $T_xM_0$ given by $\mathrm H^{f_0}(x)v=\nabla_v(\nabla f_0)$,
for all $v\in T_xM_0$. If $x$ is a critical point of $f_0$,
then $\mathfrak g_x\big(\mathrm H^{f_0}(x)v,w\big)=\mathrm d^2f_0(x)(v,w)$ is the standard second derivative
of $f_0$ at $x$.
A curve $\gamma(t)=\big(x(t),s(t)\big)$ in $M$ is a geodesic relatively to the metric \eqref{eq:formastationary} if and only if
its components $x$ and $s$ satisfy the system of differential equations:
\[\tfrac{\mathrm D}{\mathrm dt}\dot x+\tfrac{\mathrm D}{\mathrm dt}(\dot s\,\delta)-\dot s\,(\nabla\delta)^\star(\dot x)+\tfrac12\nabla\beta(x)\,\dot s^2=0,
\quad \frac{\mathrm d}{\mathrm dt}\big[\mathfrak g_x\big(\delta(x),\dot x\big)-\beta(x)\,\dot s\big]=0,\]
where $\tfrac{\mathrm D}{\mathrm dt}$ denotes covariant differentiation along $x$ relatively to the connection $\nabla$,
and $(\nabla\delta)^\star$ is the $(1,1)$-tensor on $M$ defined by $\mathfrak g\big((\nabla\delta)^\star(v),w\big)=\mathfrak g\big(\nabla_w\delta,w\big)$
for all $v,w\in TM$. As observed above, if $x_0$ is a critical point of $\beta$, i.e., $\nabla\beta(x_0)=0$, then
the curve $\gamma(t)=(x_0,t)$, $t\in[0,1]$, is a geodesic in $(M,g)$.

Let us consider for simplicity the static case, i.e., $\delta\equiv0$.
The second variation of the $g$-geodesic action functional
at a given geodesic $\gamma(t)=\big(x(t),s(t)\big)$, $t\in[0,1]$, is given by:
\begin{multline*}
I_{\mathfrak g,\beta}(\gamma)\big[(\xi,\sigma),(\bar\xi,\bar\sigma)\big]=
\int_0^1\Big[\,\mathfrak g\big(\Ddt\xi,\Ddt\bar\xi\big)
+\mathfrak g\big(R(\xi,\dot x)\bar\xi,\dot x\big)
-\bar\sigma'\,\dot s\,\mathfrak g\big(\nabla\beta(x),\xi\big)\\-
\sigma'\,\dot s\,\mathfrak g\big(\nabla\beta(x),\bar\xi\big)
-\tfrac12\dot s^2\,\mathfrak g\big(\mathrm H^\beta(x)\xi,\bar\xi\big)
-\beta(x)\sigma'\bar\sigma'\Big]\mathrm dt,
\end{multline*}
where $\xi$, $\bar\xi$ are variational vector fields along $x$ vanishing at the endpoints,
and $\sigma,\bar\sigma$ are smooth functions on $[0,1]$ vanishing at $0$ and at $1$.
In the above formula and in the rest of the section we will denote by a dot the derivatives
of the components $x$ and $s$ of the curve $\gamma$, and with a prime the derivatives of
the component $\sigma$ of the vector field $V=(\xi,\sigma)$ along $\gamma$.
A pair
$(\xi,\sigma)$ is a Jacobi field along the geodesic $\gamma=(x,s)$ if it satisfies the
second order linear system of differential equations:
\begin{equation}
\label{eq:statJacobi1}
\tfrac{\mathrm D^2}{\mathrm dt^2}\,\xi-R(\dot x,\xi)\,\dot x
+\sigma'\,\dot s\,\nabla\beta(x)+\tfrac12\dot s^2
\,\mathrm H^\beta(x)\xi=0,
\end{equation}
and
\begin{equation}\label{eq:statJacobi2}
\ddt\big[\mathfrak g\big(\dot s\,\mathfrak g\big(\nabla\beta(x),\xi\big)
+\beta(x)\,\sigma'\big]=0.
\end{equation}
In order to construct the required example, let us consider a geodesic of the form $\gamma(t)=(x_0,t)$, $t\in[0,1]$, where $x_0\in M_0$ is a critical point of $\beta$.
Equations \eqref{eq:statJacobi1} and \eqref{eq:statJacobi2} become:
\[\tfrac{\mathrm D^2}{\mathrm dt^2}\,\xi+\tfrac12\,\mathrm H^\beta(x_0)\xi=0,\quad\text{and}\quad
\sigma''=0.\]
Thus, if $V=(\xi,\sigma)$ is a Jacobi field along $\gamma$ that vanishes at $0$ and at $1$, then $\sigma\equiv0$, while
$\xi$ is a smooth curve in $T_{x_0}M_0$ satisfying the first of the two equations above.
Note that the covariant derivative $\Ddt\xi$ in this case equals the standard derivative $\xi'$.
Assume that this equation
has a non trivial solution $\xi$ satisfying $\xi(0)=\xi(1)=0$ and $\int_0^1\xi(t)\,\mathrm dt=0$.
For instance, one can take $M_0=\R$, $x_0=0$ and $\beta(x)=1+4\pi^2x^2$; then, $\tfrac12\beta''(0)=8\pi^2$,
and the differential equation $\xi''+4\pi^2\xi=0$ has the solution $\xi(t)=\sin(2\pi t)$ with the required properties.
Similar examples can be given easily in higher dimensions.

An \emph{infinitesimal variation} $h$ of $g$ \emph{in the class of stationary metrics} on $M$ of the type
\eqref{eq:formastationary} has the form:
\begin{equation}\label{eq:formastationaryinf}h_{(x,s)}\big((v,r),(\bar v,\bar r)\big)=\mathfrak h_x(v,\bar v)+\mathfrak g_x\big(\rho(x),v\big)\bar r+
\mathfrak g_x\big(\rho(x),\bar v\big)r+\zeta(x)r\bar r,\end{equation}
where $x\in M_0$, $s\in\R$, $v,\bar v\in T_xM_0$, $r,\bar r\in T_s\R\cong\R$, $\mathfrak h$ is a symmetric $(0,2)$-tensor
on $M_0$, $\rho\in\mathfrak X(M_0)$ is a smooth vector field on $M_0$, and $\zeta:M_0\to\R$ is a smooth  function
on $M_0$.
We claim that for every such $h$, the quantity $\int_0^1\big[h\big(\dot\gamma,\Ddt V)+\tfrac12\nabla h(V,\dot\gamma,\dot\gamma)\big]\,\mathrm dt$
vanishes. Namely,
\[h\big(\dot\gamma,\Ddt V\big)=\mathfrak g\big(\rho(x_0),\xi'\big),\]
and thus
\[\int_0^1h\big(\dot\gamma,\Ddt V\big)\,\mathrm dt=\mathfrak g\big(\rho(x_0),\xi(1)-\xi(0)\big)=0.\]
Moreover,
\[\nabla h(V,\dot\gamma,\dot\gamma)=\nabla_\xi\mathfrak h(\dot x,\dot x)+2\mathfrak g_0\big(\nabla_\xi\rho,\dot x)\dot s+\xi(\zeta)\dot s^2=
\xi(\zeta);\]
hence:
\[\int_0^1\nabla h(V,\dot\gamma,\dot\gamma)\,\mathrm dt=\int_0^1\xi(\zeta)\,\mathrm dt=\int_0^1\mathfrak g(\nabla\zeta,\xi)\,\mathrm dt=
\mathfrak g\Big(\nabla\zeta(x_0),\int_0^1\xi(t)\,\mathrm dt\Big)=0.\]
This proves our claim and gives the desired counterexample in the stationary case.
\end{section}
\begin{section}{Genericity in the $C^\infty$ category}\label{sec:smooth}
It is desirable to have a genericity result also in the space of $C^\infty$-metric
tensors, endowed with the Whitney weak $C^\infty$ topology (see for instance \cite{Hir}).
When the base manifold is non compact, the space of all symmetric tensors, endowed with the
topology of $C^\infty$ convergence on compact sets, is a only a \emph{Frechet} space, so that our Banach space approach does not apply directly.
However, as it was brought to the attention of the authors by the referee, there is
an elegant argument due to Taubes that allows to extend to the $C^\infty$ realm
our results. The same idea was used in \cite{FHS}, which is where the authors learned
about it; we will sketch here the argument adapted to our situation.

Consider a differentiable manifold $M$, a complete Riemannian metric $g_{\mathrm R}$ on $M$,
two distinct points $p,q\in M$, consider the sequence $\mathcal E^k=\mathbf\Gamma_{\mathrm{sym}}^k(TM^*\otimes TM^*;g_{\mathrm R})$
of $C^k$-Whitney type Banach space of tensor fields on $M$  described in Example~\ref{exa:exaWtBstf},
Subsection~\ref{sub:semiRiemannianmetrics}.
Note that the set of tensors of class $C^\infty$ having
compact support is dense in each $\mathcal E^k$. In particular, $\mathcal E^\infty=\bigcap_{k\ge k_0}\mathcal E^k$
is dense in every $\mathcal E^k$. We will think of $\mathcal E^\infty$ as a Frechet space endowed
with the family of seminorms $\Vert\cdot\Vert_k$ defined in \eqref{eq:defnormak}. In particular,
$\mathcal E^\infty$ is a \emph{Baire space}, i.e., the intersection of a countable family of dense
open subsets is dense.

Let $k_0\ge2$ be fixed, and let $\mathcal A$
be an open subset of $\mathcal E^{k_0}$ consisting of nondegenerate tensors, i.e., semi-Riemannian
metrics on $M$. For $k\ge k_0$, set $\mathcal A_k=\mathcal A\cap\mathcal E^k$; this is an open
subset of $\mathcal E^k$. Define $\mathcal A_*$ to be the subset of $\mathcal A$ consisting of all
metric tensors for which all geodesics connecting $p$ and $q$ are nondegenerate.
By assumption $\mathcal A_{k,*}=\mathcal A_*\cap\mathcal A_k$ is a generic subset of $\mathcal A_k$
for all $k\ge k_0$. Finally, define $\mathcal A_\infty=\mathcal A\cap\mathcal E^\infty=\bigcap_{k\ge k_0}\mathcal A_k\subset\mathcal E^\infty$, which is a dense
subset of $\mathcal A_k$ for all $k$,
and set $\mathcal A_{\infty,*}=\mathcal A_*\cap\mathcal A_\infty$.
Note that $\mathcal A_\infty$ is an open subset of $\mathcal E^\infty$, and thus it is also a Baire space;
convergence in $\mathcal A_\infty$ implies $C^\infty$-convergence on compact subsets of $M$.
We want to prove that $\mathcal A_{\infty,*}$
is generic in $\mathcal A_\infty$.
To this aim, denote by $\mathrm L_{\mathrm R}$  the length functional of curves relative to the Riemannian
metric $g_{\mathrm R}$;  for all $M>0$
define the following sets:
\begin{multline*}\mathcal A_{k,*,M}=\\\big\{g\in\mathcal A_k:\text{all $g$-geodesic $\gamma$ connecting $p$ and $q$,
with $\mathrm L_{\mathrm R}(\gamma)\le M$, are nondegerate}\big\},\end{multline*}
and
\[\mathcal A_{\infty,*,M}=\bigcap_{k\ge k_0}\mathcal A_{k,*,M}.\]
Clearly, $\mathcal A_{\infty,*}=\bigcap_{M=1}^\infty\mathcal A_{\infty,*,M}$, thus, to prove
our claim it suffices to show that $\mathcal A_{\infty,*,M}$ is open and dense in $\mathcal A_\infty$.
The key observation is that for all $k$ and $M$, $\mathcal A_{k,*,M}$ is open in $\mathcal A_k$.
This follows from the following argument. Assume that $g_n\in\mathcal A_k\setminus\mathcal A_{k,*,M}$
is a sequence converging to some $g_\infty\in\mathcal A_k$. Then, there exists a sequence $\gamma_n:[0,1]\to M$
of $g_n$-geodesics connecting $p$ and $q$, with $\mathrm L_{\mathrm R}(\gamma_n)\le M$ for all $n$,
and such that there is a non trivial $g_n$-Jacobi field $J_n$ along $\gamma_n$ with $J_n(0)=J_n(1)=0$
for all $n$. Each $J_n$ can be normalized so that
\begin{equation}\label{eq:normalization}
\left\Vert\frac{\mathrm D^{g_n}}{\mathrm dt}J_n(0)\right\Vert=1
\end{equation}
for all $n$; here $\frac{\mathrm D^{g_n}}{\mathrm dt}J_n$ is the covariant derivative of $J_n$ along $\gamma_n$
relatively to the Levi--Civita connection of $g_n$. Using the completeness of $g_{\mathrm R}$, by the theorem
of Arzel\'a and Ascoli, we can assume that the sequence $\gamma_n$ converges to a curve $\gamma_\infty$ connecting
$p$ and $q$; an immediate continuity argument shows that $\gamma_\infty$ is a $g_\infty$-geodesic with
$\mathrm L_{\mathrm R}(\gamma_\infty)\le M$. By \eqref{eq:normalization}, we can also assume that
the sequence $v_n=\frac{\mathrm D^{g_n}}{\mathrm dt}J(0)\in T_pM$ is convergent to some $v_\infty\ne0$.
By continuity, the $g_\infty$-Jacobi field $J_\infty$ along $\gamma_\infty$ satisfying $J_\infty(0)=0$
and $\frac{\mathrm D^{g_\infty}}{\mathrm dt}J_\infty(0)=v_\infty$ also satisfies $J_\infty(1)=0$, i.e.,
$\gamma_\infty$ is a degenerate $g_\infty$-geodesic connecting $p$ and $q$, and $g_\infty\not\in\mathcal A_{k,*,M}$.
This shows that $\mathcal A_{k,*,M}$ is open in $\mathcal A_k$  for every $M>0$ and $k\in\N\cup\{+\infty\}$.
Moreover, since $\mathcal A_{k,*,M}$ contains $\mathcal A_{k,*}$, then $\mathcal A_{k,*,M}$ is also dense
in $\mathcal A_k$ for all $M$. Finally, since $\mathcal A_\infty$ is dense in $\mathcal A_k$ and
$\mathcal A_{k,*,M}$ is open and dense in $\mathcal A_k$, then $\mathcal A_\infty\cap\mathcal A_{k,*,M}=\mathcal A_{\infty,*,M}$
is dense in $\mathcal A_k$ for all $k$, and thus $\mathcal A_{\infty,*,M}$ is dense in $\mathcal A_\infty$.
This proves the genericity result in the $C^\infty$-category. Analogous results hold in all the cases discussed
in Section~\ref{sec:morsegeofunctionals}.
\end{section}
\begin{section}{A few final remarks and open problems}
\label{sec:remarks}
Let us conclude with a few observations.

First, one should observe that the genericity result for globally hyperbolic Lorentzian
manifolds stated in Subsection~\ref{sub:globhyper} is far from being conclusive, or exhaustive.
Note  for instance that Proposition~\ref{thm:genglobhyperb} does not apply to sets containing
metric tensors $g^{\alpha,\beta}$ with $\beta$ an unbounded function on $M$.
Several different statements of the genericity result are possible by the very same
argument, simply by selecting the appropriate set of tensors and its Banach space structure
that one wants to consider. It should also be mentioned that somewhat stronger genericity
results may be obtained by relaxing the global hyperbolicity condition given in
\eqref{eq:garantisceglobhyerb}, in that the inequality may be required to hold in
smaller regions of the spacetime. For instance, in \cite{AbbMej} it is given a condition
on the first derivative of the metric coefficients $\alpha$ and $\beta$ implying that
all the geodesics between the prescribed points $p$ and $q$ remain in a time-limited region
of the spacetime.
However, such stronger results would
certainly have a more involved statement, filled with technicalities that are probably not
appropriate for the purposes of the present paper. The interested reader will have no problem
in adapting the arguments in the proof of Proposition~\ref{thm:genglobhyperb} to other specific
cases.

As to the stationary Lorentzian case (Subsection~\ref{sec:stationary}), the \emph{negative} result
given by the counterexample exhibited opens several interesting questions and conjectures that
deserve further attention. First, it is natural to conjecture that, apart from vertical geodesics,
stationary infinitesimal perturbations of the metric would suffice to destroy degeneracies.
Should this be the case, than a genericity result may be obtained by considering points $p$ and
$q$ that do not belong to the same integral line of the Killing vector field. A proof
for the existence of appropriate infinitesimal perturbations would have to based on the following
conjecture: given a non vertical geodesic $\gamma=(x,s)$ and a nontrivial Jacobi field
$J=(\xi,\tau)$ along $\gamma$ vanishing at the endpoints, then at some instants $t$, the vector $\xi(t)$ is
not parallel to $\dot x(t)$. A direct proof of this fact, based on the Jacobi differential equations
\eqref{eq:statJacobi1} and \eqref{eq:statJacobi2}, seems to be rather involved, so that a suitable
version of Lemma~\ref{thm:Jnontparalleltogamma} would have to be proven.
Another interesting point would be to determine the genericity of the nondegeneracy property
in the stationary Lorentzian case if one allows that also the Killing
vector field $Y$ may be perturbed. We conjecture that the genericity property in this case
would hold under no restrictions on the endpoint.

Finally, we would like to mention the case of closed geodesics, which is substantially
more involved than the fixed endpoint case. Let us recall that the first statement
of the Riemannian bumpy theorem is due to Abraham, see \cite{Abr}, but to the authors'
knowledge the first complete proof of it is due to Anosov, see \cite{Ano}.
A very interesting observation is
that a similar result does not hold for a general conservative Hamiltonian system, where one can have degenerate
periodic orbits that are not destroyed by small perturbations, as shown in \cite{MeyPal}.
Significative improvements of the bumpy metric theorem
have been proven later by Klingenberg and Takens \cite{KliTak}, who have shown genericity of the set
of metrics with the property that the Poincar\'e map of every closed geodesic and all its derivatives up to
a finite order belong to a prescribed open and dense subset of the space of jets of symplectic maps
around a fixed point.

As we have observed, the theory developed in this paper does not work in order
to prove a genericity result for closed geodesics: iterates cannot
be dealt with the perturbation arguments discussed. Although parts of Anosov's proof
of the bumpy metric theorem in \cite{Ano} can be carried over to the semi-Riemannian
case (namely, all the properties depending on the linearized Poincar\'e map), the positive
definite character of Anosov's argument in some parts of the proof cannot be extended
directly to the semi-Riemannian case. For instance, it is used in \cite{Ano} a certain
lower bound on the length of closed geodesics for all Riemannian metrics in a neighborhood of a given
one; such bound certainly does not exist outside the Riemannian realm. A natural conjecture,
or more exactly a wishful thinking at this stage, is that bumpy metrics may be generic in
sets of Lorentzian metrics satisfying restrictive causality and geometric assumptions.
A natural guess would be starting with the stationary and globally hyperbolic case, where
all closed geodesics are spacelike, and recent developments of the variational geodesic
theory (refs.\ \cite{BilMerPic, CanFloSan}) indicate a certain Riemannian behavior of the geodesic flow.
\end{section}


\begin{thebibliography}{99}

\bibitem{AbbMej2} \textsc{A. Abbondandolo, P. Majer}, \emph{A Morse complex for infinite dimensional manifolds. I}, Adv.\ Math.\ \textbf{197} (2005), no.\ 2, 321--410

\bibitem{AbbMej} \textsc{A. Abbondandolo, P. Majer}, \emph{A Morse complex for Lorentzian geodesics}, Preprint (2006),
arXiv : \texttt{math/0605261v1[math.DG]}.

\bibitem{AbboMaj3}
{\sc A.~Abbondandolo and P.~Majer}, {\em Lectures on the {M}orse complex for
  infinite-dimensional manifolds}, in Morse theoretic methods in nonlinear
  analysis and in symplectic topology, vol.~217 of NATO Sci. Ser. II Math.
  Phys. Chem., Springer, Dordrecht, 2006, pp.~1--74.

\bibitem{Abr}
\textsc{R.~Abraham}, \emph{Bumpy metrics}, in Global Analysis (Proc. Sympos. Pure
  Math., Vol. XIV, Berkeley, Calif., 1968), Amer. Math. Soc., Providence, R.I.,
  1970, pp.~1--3.

\bibitem{AbrRob} \textsc{R. Abraham, J. Robbin}, \emph{Transversal mappings and flows}, W. A. Benjamin, New York, 1967.

\bibitem{Ano} \textsc{D. V. Anosov}, \emph{Generic properties
of closed geodesics},  Izv.\ Akad.\ Nauk SSSR
Ser.\ Mat.\ \textbf{46} (1982), no.\ 4, 675--709, 896.

\bibitem{BeErEa96}
{\sc J.~K. Beem, P.~E. Ehrlich, and K.~L. Easley}, {\em Global {L}orentzian
  geometry}, vol.~202 of Monographs and Textbooks in Pure and Applied
  Mathematics, Marcel Dekker Inc., New York, second~ed., 1996.

\bibitem{BerSan2} \textsc{A. Bernal, M. S\'anchez}, \emph{On smooth Cauchy hypersurfaces and Geroch's splitting theorem},
Comm.\ Math.\ Phys.\ \textbf{243} (2003), no.\ 3, 461--470.

\bibitem{BerSan3} \textsc{A. Bernal, M. S\'anchez}, \emph{Smoothness of time functions and the metric
splitting of globally hyperbolic spacetimes},
Comm.\ Math.\ Phys.\ \textbf{257} (2005), no.\ 1, 43--50.

\bibitem{BeSan2} \textsc{A. Bernal, M. S\'anchez}, \emph{Globally hyperbolic spacetimes can be defined as `causal'
instead of `strongly causal'},
Class.\ Q.\ Grav.\ \textbf{24} (2007), no.\ 3, 745--749.


\bibitem{BilMerPic}
\textsc{L.~Biliotti, F.~Mercuri, and P.~Piccione}, \emph{On a Gromoll--Meyer type
  theorem in globally hyperbolic stationary spacetimes}, Comm.\ Anal.\ Geom.\ \textbf{16} (2008), no.\ 2, 333--393.


\bibitem{CanFloSan}  \textsc{A. M. Candela, J. L. Flores, M. S\'anchez},
\emph{Global hyperbolicity and Palais-Smale condition for action functionals in stationary spacetimes },
Adv.\ Math.\ \textbf{218} (2008), no.\ 2, 515--536.


\bibitem{DaNo}
\textsc{M. Dajczer, K. Nomizu}, {\em On the boundedness of Ricci curvature of an indefinite metric}, Bol. Soc. Brasil. Mat. \ \textbf{11} (1980), n.\ 1, pp.25--30.

\bibitem{FHS} {\sc A.~Floer, H.~Hofer, and D.~Salamon}, {\em Transversality in elliptic
 {M}orse theory for the symplectic action}, Duke Math.\ J.\ \textbf{80} (1995),
251--292.

\bibitem{Ger} \textsc{R. Geroch}, \emph{Domain of dependence}, J.\ Math.\ Phys.\ \textbf{11}, 437--449 (1970).

\bibitem{CAG1}
\textsc{F.~Giannoni and P.~Piccione}, {\em An intrinsic approach to the geodesical
  connectedness of stationary {L}orentzian manifolds}, Comm.\ Anal.\ Geom.\ \textbf7
  (1999), pp.~157--197.

\bibitem{GroMey2}
\textsc{D.~Gromoll and W.~Meyer}, {\em Periodic geodesics on compact riemannian
  manifolds}, J. Differential Geometry, 3 (1969), pp.~493--510.

\bibitem{Hir} \textsc{M. W. Hirsch}, \emph{Differential topology},
Graduate Texts in Mathematics, 33. Springer-Verlag, New York, 1994.

\bibitem{JavPic} \textsc{M. A. Javaloyes, P. Piccione}, \emph{On the singularities of the semi-Riemannian exponential map. Bifurcation of geodesics and light rays},
Variations on a century of relativity: theory and applications, 115--123, Lect.\ Notes Semin.\ Interdiscip.\ Mat., V, S.I.M.\ Dep.\ Mat.\ Univ.\ Basilicata, Potenza, 2006.

\bibitem{Kli}
\textsc{W.~Klingenberg}, {\em Lectures on closed geodesics}, Springer-Verlag,
  Berlin, 1978.
\newblock Grundlehren der Mathematischen Wissenschaften, Vol. 230.

\bibitem{KliTak}
\textsc{W.~Klingenberg and F.~Takens}, {\em Generic properties of geodesic flows},
  Math.\ Ann.\ \textbf{197} (1972), pp.~323--334.


\bibitem{MeyPal} \textsc{K. R. Meyer, J.  Palmore},
\emph{A generic phenomenon in conservative Hamiltonian systems},
1970 Global Analysis (Proc.\ Sympos.\ Pure Math., Vol.\ XIV, Berkeley, Calif., 1968) pp.\ 185--189.

\bibitem{MinSan} \textsc{E. Minguzzi, M. Sanchez}, \emph{The causal hierarchy of spacetimes}, Preprint (2006), arXiv :
\texttt{gr-qc/0609119v2}.

\bibitem{One83}
\textsc{B.~O'Neill}, {\em Semi-{R}iemannian geometry}, vol.~103 of Pure and
  Applied Mathematics, Academic Press Inc. [Harcourt Brace Jovanovich
  Publishers], New York, 1983.
\newblock With applications to relativity.

\bibitem{SanBari} \textsc{M. S\'anchez}, \emph{Some remarks on causality theory and variational methods in Lorenzian manifolds},
Conf.\ Semin.\ Mat.\ Univ.\ Bari No.\ 265 (1997).

  \bibitem{Sma} \textsc{S. Smale}, \emph{An infinite dimensional version of Sard's theorem},
  Amer.\ J.\ Math.\ \textbf{87} (1965), 861--866.

\bibitem{Whi}
\textsc{B.~White}, {\em The space of minimal submanifolds for varying {R}iemannian
  metrics}, Indiana Univ.\ Math.\ J.\ \textbf{40} (1991), pp.~161--200.

\end{thebibliography}
\end{document}